\documentclass[a4paper,10pt
]{amsart}
\usepackage{amsmath, amscd}
\usepackage{amssymb}
\usepackage[alphabetic]{amsrefs}
\usepackage[T1]{fontenc}
\usepackage[all]{xy}
\usepackage{mathtools}
\usepackage{footmisc}
\usepackage{tikz-cd}
\usepackage{relsize} 
\usepackage[bbgreekl]{mathbbol} 
\usepackage{amsfonts} 
\usepackage[colorlinks=true, hyperindex, linkcolor=magenta, citecolor=blue, pagebackref=false, urlcolor=blue, linktocpage]{hyperref}
\usepackage[color = blue!20, bordercolor = black, textsize = tiny]{todonotes}

\usepackage[left=3cm, right=3cm, top = 2.8cm, bottom = 2.8cm, marginparwidth=2.4cm ]{geometry}

\makeatletter
\def\@tocline#1#2#3#4#5#6#7{\relax
  \ifnum #1>\c@tocdepth 
  \else
    \par \addpenalty\@secpenalty\addvspace{#2}%
    \begingroup \hyphenpenalty\@M
    \@ifempty{#4}{%
      \@tempdima\csname r@tocindent\number#1\endcsname\relax
    }{%
      \@tempdima#4\relax
    }%
    \parindent\z@ \leftskip#3\relax \advance\leftskip\@tempdima\relax
    \rightskip\@pnumwidth plus4em \parfillskip-\@pnumwidth
    #5\leavevmode\hskip-\@tempdima
      \ifcase #1
      \or\or \hskip 2em \or \hskip 2em \else \hskip 3em \fi%
      #6\nobreak\relax
    \dotfill\hbox to\@pnumwidth{\@tocpagenum{#7}}\par
    \nobreak
    \endgroup
  \fi}
\makeatother

\DeclareSymbolFontAlphabet{\mathbb}{AMSb} 
\DeclareSymbolFontAlphabet{\mathbbl}{bbold} 
\newcommand{\Prism}{{\mathlarger{\mathbbl{\Delta}}}}

\numberwithin{equation}{section}
\newtheorem{theorem}[subsection]{Theorem}
\newtheorem{corollary}[subsection]{Corollary}
\newtheorem{lemma}[subsection]{Lemma}
\newtheorem{proposition}[subsection]{Proposition}
\theoremstyle{definition}
\newtheorem{definition}[subsection]{Definition}
\newtheorem{remark}[subsection]{Remark}
\newtheorem{example}[subsection]{Example}

\newtheorem{construction}[subsection]{Construction}

\newcommand{\F}{\mathbb{F}}

\renewcommand{\L}{\mathbb{L}}

\newcommand{\N}{\mathbb{N}}

\newcommand{\Q}{\mathbb{Q}}

\newcommand{\Z}{\mathbb{Z}}

\newcommand{\cV}{\mathcal{V}}

\newcommand{\cal}{\mathcal}

\DeclareMathOperator{\THH}{THH}
\DeclareMathOperator{\HH}{HH}

\newcommand{\gp}{\mathrm{gp}}
\newcommand{\rep}{\mathrm{rep}}
\newcommand{\GL}{\mathrm{GL}}

\newcommand{\perf}{\mathrm{perf}}
\newcommand{\lQRSPerfd}{\mathrm{lQRSPerfd}}

\newcommand{\lQSyn}{\mathrm{lQSyn}}
\newcommand{\QSyn}{\mathrm{QSyn}}

\title[]{Logarithmic Prismatic Cohomology \\ via Logarithmic THH}

\author{Federico Binda}
\address{Department of Mathematics ``F. Enriques'', University of Milan, Italy}
\email[F. Binda]{federico.binda@unimi.it}

\author{Tommy Lundemo}
\address{Department of Mathematics and Informatics, University of Wuppertal, Germany}
\email[T. Lundemo]{lundemo@math.uni-wuppertal.de}

\author{Doosung Park}
\address{Department of Mathematics and Informatics, University of Wuppertal, Germany}
\email[D. Park]{dpark@uni-wuppertal.de}

\author{Paul Arne {\O}stv{\ae}r}
\address{Department of Mathematics ``F. Enriques'', University of Milan, Italy \&
Department of Mathematics, University of Oslo, Norway}
\email[P.A. {\O}stv{\ae}r]{paul.oestvaer@unimi.it \& paularne@math.uio.no}

\begin{document}
\maketitle

\begin{abstract} 
Inspired by Bhatt--Morrow--Scholze's work on ${\rm THH}$, 
we define Nygaard-completed log prismatic cohomology based on log topological Hochschild homology 
via filtrations on log ${\rm THH}$ and its variants. 
Moreover, 
of independent interest, 
we describe log ${\rm THH}$ for quasiregular semiperfectoid pre-log rings as a $1$-parameter deformation of ordinary, 
non-logarithmic Hochschild homology.
\end{abstract}
\setcounter{tocdepth}{1}

\tableofcontents

\section{Introduction}

The work of Bhatt--Morrow--Scholze \cite{BMS19} provides filtrations on variants of topological Hochschild homology, 
in analogy with the motivic filtrations on algebraic $K$-theory \cite{FS02}, 
\cite{levine.coniveau}, 
\cite{voevodsky.open}. 
Control over the graded pieces allows one to define new integral cohomology theories,
such as prismatic and syntomic cohomology, for $p$-adic schemes. 
These are computable invariants:
Antieau--Krause--Nikolaus \cite{AKN22} reduces the computation of topological cyclic homology 
(and hence algebraic $K$-theory \cites{Qui72, DGM13}) 
of rings like ${\Bbb Z}/p^n$ to the computation of prismatic cohomology, 
for which they provide an algorithm.

An example more in the spirit of the present paper is the original arithmetic application of \cite{BMS19}, 
in which \'etale, crystalline, and de Rham cohomology are exhibited as specializations of a common 
cohomology theory \cite[Theorem 1.2]{BMS19}. 
The results of this paper provide a first step in establishing matching comparison results in the  context of 
logarithmic geometry using  techniques from homotopy theory (see \cite{KY22} for a different approach). The interest in such results is justified because a much wider class of objects behaves 
``as if they were smooth'' in logarithmic geometry.  
Standard examples are varieties with semi-stable reduction.

\subsection{Log topological Hochschild homology} 
We begin by quickly reviewing the construction of  (Nygaard-complete) prismatic cohomology, following \cite{BMS19}. 
It is a $p$-adic cohomology theory built in several steps, 
starting with the case of \emph{quasiregular semiperfectoid} rings: this is a convenient class of rings that contains perfectoid rings, but also the tensor products of perfections of smooth algebras over a field of positive characteristic. 
In this case, 
one defines 
\[\widehat{\Prism}_S 
:= 
\pi_0{\rm TC}^-(S ; {\Bbb Z}_p),
\] 
where ${\rm TC}^-(S ; {\Bbb Z}_p) := {\rm THH}(S ; {\Bbb Z}_p)^{hS^1}$ is the homotopy fixed points of the 
natural $S^1$-action on topological Hochschild homology \cite{zbMATH07303328}. 
One feature of quasiregular semiperfectoid rings is that ${\rm THH}(S ; {\Bbb Z}_p)$ is \emph{even}, 
i.e., its odd homotopy groups vanish, 
and $\widehat{\Prism}_S$ is complete for the filtration of the degenerate homotopy fixed point 
spectral sequence 
\[E_2^{s, t} = H^s(S^1, \pi_t{\rm THH}(S ; {\Bbb Z}_p)) \implies \pi_{t - s} {\rm TC}^-(S ; {\Bbb Z}_p).\] 
By definition, this is the \emph{Nygaard filtration} on $\widehat{\Prism}_S$. 
The relationship between this construction and the site-theoretic definition of prismatic cohomology 
$\Prism_S$ is the content of \cite[Theorem 13.1]{BS22}, 
where Bhatt and Scholze exhibit $\widehat{\Prism}_S$ as the completion of $\Prism_S$ with respect to an 
independently defined Nygaard filtration, thus justifying a posteriori the terminology.  The \emph{unfolding} equivalence of \cite[Proposition 4.31]{BMS19} extends the definition of $\widehat{\Prism}_A$ from quasiregular semiperfectoid rings to quasisyntomic rings. 
\medskip

The key construction carried out in this paper is that of a logarithmic version of the Nygaard-complete 
prismatic cohomology $\widehat{\Prism}_S$ of \cite{BMS19}. 
For this, 
we shall replace topological Hochschild homology with \emph{log topological Hochschild homology} 
in the sense of \cite{Rog09}. 

Let us review the basics of this construction. 
A \emph{pre-log ring} $(A, M, \alpha)$ consists of a commutative ring $A$, 
a commutative monoid $M$, 
and a map of commutative monoids $\alpha \colon M \to (A, \cdot)$. 
In this paper, 
pre-log rings will play the role of affine schemes in logarithmic geometry \cite{Kat89}; 
we refer to \cite[Section 7.4]{BLPO} for further discussion on this point.  
The \emph{log ${\rm THH}$ of $(A, M)$}, which we review in Section \ref{sec:logthh}, 
is a commutative $A$-algebra in spectra \cite[Definition 8.11]{Rog09}.
As we explain in Section \ref{subsec:cyclotomic}, 
log ${\rm THH}$ carries a cyclotomic structure in the sense of Nikolaus--Scholze \cite{NS18}, 
and hence we obtain definitions of log ${\rm TC}^-, {\rm TP}$, and ${\rm TC}$. 

One goal of this paper is to define logarithmic analogs of some key notions in \cite{BMS19}, 
such as quasisyntomic and quasiregular semiperfectoid, 
that allow us to define Nygaard-completed log prismatic cohomology in a manner analogous to that of \emph{loc.\ cit.}
While this leads to analogs of the results of \cite{BMS19}, it also sheds new light on logarithmic ${\rm THH}$. 
As we explain in further detail in Section \ref{subsec:deformlogthhtohh}, 
the results of this paper show that log ${\rm THH}$ can be thought of as a $1$-parameter deformation of 
ordinary, non-logarithmic Hochschild homology for quasiregular semiperfectoid pre-log rings.

\subsection{A log variant of a result of Antieau}
\label{subsec:logantieau}  
Perhaps the easiest example of a filtration in the spirit of \cite{BMS19} is the 
\emph{Hochschild--Kostant--Rosenberg-filtration} ${\rm Fil}_{\rm HKR}^{\bullet} {\rm HH}(A / R)$ on ordinary
Hochschild homology ${\rm HH}(A/R)$, 
with graded pieces given by the derived exterior powers of the algebraic cotangent complex ${\Bbb L}_{A / R}$. 
We obtained in \cite{BLPO} a similar filtration for log Hochschild homology and Gabber's cotangent complex.
These constructions are briefly reviewed in Section \ref{sec:loghh}.

Let $R \to A$ be a map of commutative rings. 
Generalizing a result of Bhatt--Morrow--Scholze \cite[Theorem 1.17]{BMS19} from the $p$-complete setting, 
Antieau \cite{Ant19} constructed a filtration on periodic cyclic homology ${\rm HP}(A / R)$ with graded pieces 
given by shifts of the Hodge-completed derived de Rham complex $\widehat{\Omega}_{A / R}$. 

Antieau's approach is appealing due to its generality; the only necessary input specific to the case of (the $S^1$-Tate construction of) Hochschild homology are 

\begin{enumerate}
\item the ${\rm HKR}$-theorem, which gives rise to the complete, exhaustive filtration $F^{\bullet}_{\rm HKR} {\rm HC}^-(A / R)$, called the \emph{HKR-filtration}, on negative cyclic homology;
\item the compatibility of the differential in the resulting chain complexes \cite[Example 2.4]{Ant19} \[0 \xrightarrow{} \Omega^n_{A / R} \xrightarrow{} \Omega^{n + 1}_{A / R} \xrightarrow{} \cdots\] with the de Rham differential; and
\item fpqc descent for Hochschild homology and its variants, in order to extend the result from the affine case to any qcqs $R$-scheme. 
\end{enumerate}

Due to our earlier work \cite{BLPO} and the results of the present paper, we now have analogs of all of the above facts in the context of logarithmic geometry. More specifically, if $(R, P) \to (A, M)$ is a map of pre-log rings, then

\begin{enumerate}
\item there is an HKR-theorem for log Hochschild homology ${\rm HH}((A, M) / (R, P))$, which gives rise to the complete, exhaustive \emph{HKR-filtration} \[F^\bullet_{\rm HKR} {\rm HC}^-((A, M) / (R, P))\] on log negative cyclic homology (\cite[Theorem 1.1]{BLPO}, see Theorem \ref{thm:loghkrfilt});
\item the chain complex resulting from the construction of \cite[Example 2.4]{Ant19} is of the form \[0 \to \Omega^n_{(A, M) / (R, P)} \xrightarrow{} \Omega^{n + 1}_{(A, M) / (R, P)} \xrightarrow{} \cdots\] and the differential is the log de Rham differential (Proposition \ref{prop:logderham}); and
\item a descent property for log Hochschild homology and its variants
allows us to extend the result from the ``log affine'' case under suitable hypotheses 
(Proposition \ref{prop:loghhflatdescent}).
\end{enumerate}

We refer to \cite[Section 7.4]{BLPO} for further discussion of the globalization procedure since we only construct 
the filtration in the ``log affine'' affine case for the present exposition.
Defining derived log de Rham cohomology as in \cite[Section 6]{Bha12}, 
Antieau's argument applies to obtain the following result, 
whose proof we also summarize at the end of Section \ref{sec:loghh}:

\begin{theorem}\label{thm:logantieau} Let $(R, P) \to (A, M)$ be a map of pre-log rings. The log negative and periodic cyclic homologies \[{\rm HC}^-((A, M) / (R, P)), \quad {\rm HP}((A, M) / (R, P))\] admit functorial, complete, decreasing filtrations with graded pieces equivalent to

\[\widehat{L\Omega}^{\ge n}_{(A, M) / (R, P)}[2n], \quad \widehat{L\Omega}_{(A, M) / (R, P)}[2n],\] 
respectively. Each filtered piece is itself equipped with a compatible filtration, which induces the Hodge filtration on the graded pieces under the equivalences above. 

The filtrations on  ${\rm HC}^-((A, M) / (R, P))$ and ${\rm HP}((A, M) / (R, P))$ are exhaustive if the Gabber cotangent complex ${\Bbb L}_{(A, M) / (R, P)}$ has ${\rm Tor}$-amplitude in $[-1, 0]$. 
\end{theorem}

We will give a full, independent proof of Theorem \ref{thm:logantieau} in the $p$-complete setting in Section \ref{sec:logderham}, following the proof strategy of \cite[Section 5]{BMS19}. 

\subsection{The log quasisyntomic site} 
When constructing filtrations on variants of log \emph{topological} Hochschild homology, 
it is convenient to work with the log quasisyntomic site for pre-log rings. 
Its definition and main properties are discussed in Section \ref{ssec:QSynlog}. 
For $(A, M)$ a log quasisyntomic pre-log ring, the log quasisyntomic covers turns 
the category of log quasisyntomic $(A, M)$-algebras into a site.
This construction is inspired by and analogous to the non-log situation in \cite{BMS19}. 

There is a convenient basis for the topology on the log quasisyntomic site given by 
\emph{log quasiregular semiperfectoid rings} (see Definition \ref{def:logqrsperfd}). 
As in the non-log setting, 
there exists a surjection from some perfectoid ring onto the underlying commutative ring of 
every log quasiregular semiperfectoid ring $(S, Q)$.
In addition, the commutative monoid $Q$ is semiperfect;
its $p$-power map is surjective.

Examples of log quasisyntomic morphisms include 
morphisms that are smooth and lci in the eyes of the Gabber cotangent complex. Under integrality hypotheses, this recovers the usual notion of log smooth, see e.g.\ \cite[Section 4]{BLPO}. 
A version of the log quasisyntomic site is also independently pursued in  the work of 
Koshikawa--Yao \cite{KY22}: up to minor choices of technical conventions, the constructions agree.

\subsection{Deforming log ${\rm THH}$ to ordinary Hochschild homology}
\label{subsec:deformlogthhtohh}
One important step in constructing the filtrations of \cite{BMS19} is analyzing the topological Hochschild homology of 
quasiregular semiperfectoid rings $S$. In turn, this is achieved using the cofiber sequence 
\begin{equation}\label{cofiberintro}{\rm THH}(S ; {\Bbb Z}_p)[2] \xrightarrow{u} {\rm THH}(S ; {\Bbb Z}_p) \xrightarrow{} {\rm HH}(S / R ; {\Bbb Z}_p)\end{equation}
of \cite[Theorem 6.7]{BMS19}, exhibiting ${\rm THH}$ as a $1$-parameter deformation  of \emph{relative} ${\rm HH}$, where $u \in \pi_2{\rm THH}(R ; {\Bbb Z}_p)$ is a generator for $\pi_*{\rm THH}(R ; {\Bbb Z}_p) \cong R[u]$, 
with $R \to S$ a chosen surjection from some perfectoid ring $R$.  

If $(S, Q)$ is a log quasiregular semiperfectoid ring, 
there exists a surjection $R\to S$, 
where $R$ is a perfectoid ring. 
The conditions on the underlying monoid $Q$ imply that the canonical map from its \emph{tilt} $Q^\flat$ to $Q$ 
is a surjection. 
In particular, 
the canonical map $R\langle Q^\flat \rangle \to S$ is a surjection from a perfectoid ring, 
which can be used to define $u$ in the cofiber sequence \eqref{cofiberintro}.
Here $R\langle Q^\flat \rangle$ is a standard notation for the $p$-complete version of the monoid ring 
$R[Q^{\flat}]$.

\begin{theorem}\label{thm:logthhdefhh} Let $(S, Q)$ be a log quasiregular semiperfectoid pre-log ring. There is a cofiber sequence \[{\rm THH}((S, Q) ; {\Bbb Z}_p)[2] \xrightarrow{u} {\rm THH}((S, Q) ; {\Bbb Z}_p) \xrightarrow{} {\rm HH}(S / R\langle Q^{\flat, \rm rep} \rangle ; {\Bbb Z}_p),\] where $Q^{\flat, \rm rep} \to Q$ is the repletion (or exactification) of the map $Q^\flat \to Q$. 
\end{theorem}

The above result serves as evidence that we have captured the correct notions in the logarithmic context. Indeed, as we will see in the proof, it is equivalent to the log analog \[{\rm THH}((S, Q) ; {\Bbb Z}_p)[2] \xrightarrow{u} {\rm THH}((S, Q) ; {\Bbb Z}_p) \xrightarrow{} {\rm HH}((S, Q) / R \langle Q^\flat \rangle ; {\Bbb Z}_p)\] of the cofiber sequence \eqref{cofiberintro}, which is in fact a cofiber sequence for any algebra over a perfectoid ring. 

More importantly, Theorem \ref {thm:logthhdefhh} lets us view ${\rm THH}((S, Q) ; {\Bbb Z}_p)$ as a $1$-parameter deformation 
of \emph{ordinary}, non-logarithmic Hochschild homology relative to the base ring $R\langle Q^{\flat, \rm rep} \rangle$ which 
need not be a perfectoid ring.
In some sense, 
the failure of this ring being perfectoid measures the difference between the theory developed here and its 
non-logarithmic analog, see Remark \ref{rmk:comment_Qflat_to_Q_exact}.

The techniques leading to Theorem \ref{thm:logthhdefhh} appear throughout the paper: 
See, for instance, 
Proposition \ref{prop:qrqsrep}, Proposition \ref{prop:loghhvsordinaryhh}, 
and in particular Proposition \ref{prop:filterlogthhcotcx}, which gives a filtration on log ${\rm THH}$ 
with graded pieces described in terms of the 
ordinary, non-logarithmic cotangent complex. 
We refer to Examples \ref{exa:qrsperfd} and \ref{exa:postrepletion} for an example of the repletion procedure 
in this context. 
\medskip

There are several variants of topological Hochschild homology, such as log THH in \cite{Rog09} and the approaches to $\rm THH$ using Waldhausen categories by Hesselholt--Madsen \cite{HM03} and Blumberg--Mandell \cite{BM08}. Among these, Rognes' approach is unique in the sense that it manipulates the \emph{output} of the cyclic bar construction (by means of repletion) as opposed to its input. This has also served as the main philosophical obstacle in relating log THH to $K$-theoretic invariants. For this reason, we find it very appealing that the repletion procedure appears in the \emph{input} of Hochschild homology in the cofiber term in Theorem \ref{thm:logthhdefhh}.
\medskip

Variants of the commutative ring $R\langle Q^{\flat, \rm rep} \rangle$ naturally appear in other contexts. 
For example, if $S$ is a quasiregular semiperfect ring,
by \cite[Theorem 8.17]{BMS19} there is a natural isomorphism 
\[\widehat{\Prism}_S \cong \widehat{\Bbb A}_{\rm crys}(S)\] 
relating the Nygaard-complete prismatic cohomology $\widehat{\Prism}_S$ and the
(Nygaard-completion of the $p$-adic completion of) the divided power envelope of Fontaine's 
period map $W(S^\flat)\to S$. 
The above is the key step in the crystalline comparison of \cite[Theorem 1.2(3)]{BMS19}. 
In future work,
Diao--Yao \cite{DY23} defines ${\Bbb A}_{\rm crys}(S, Q)$ as
the divided power envelope of $W(S^\flat)\langle Q^{\flat, \rm rep} \rangle \to S$. 
For this reason, 
the techniques leading to Theorem \ref{thm:logthhdefhh} will be  useful in establishing a crystalline comparison 
in our context. 
Moreover, as explained in Remark \ref{rem:koshikawayao}, 
this should also prove useful in comparing our construction with the future work of Koshikawa--Yao, 
in which log prismatic cohomology is developed using Koshikawa's log prismatic site \cite{Kos22}. 
We hope to relate our construction to the Nygaard-completion of Koshikawa--Yao's derived log prismatic cohomology.

\subsection{The log filtrations} The log variants of quasiregular semiperfectoid and quasisyntomic enjoy many formal properties similar to those of their non-log counterparts considered in \cite{BMS19}. For example, there is an \emph{unfolding} equivalence, which identifies sheaves defined on log quasiregular semiperfectoid pre-log rings with sheaves defined on the log quasisyntomic site. See Theorem \ref{thm:unfolding}. We denote the latter site by ${\rm lqsyn}$. 
For example, the functor $(S, Q) \mapsto \pi_0{\rm TC}^-((S, Q) ; {\Bbb Z}_p)$ is a sheaf defined on log quasiregular semiperfectoid rings, and hence unfolds to a sheaf on ${\rm lqsyn}$. 

In analogy with the definition of $\widehat{\Prism}_A$ in \cite{BMS19}, 
we define the \emph{Nygaard-complete log prismatic cohomology} of a log quasisyntomic pre-log ring $(A, M)$ by
\[\widehat{\Prism}_{(A, M)} := R\Gamma_{{\rm lqsyn}}((A, M), \pi_0{\rm TC}^-((-, -) ; {\Bbb Z}_p)).\] The homotopy fixed points spectral sequence computing the group 
$\pi_0{\rm TC}^-((A, M) ; {\Bbb Z}_p)$ gives rise to a 
complete filtration ${\rm Fil}_N^{\ge \bullet} \widehat{\Prism}_{(A, M)}$ of $\widehat{\Prism}_{(A, M)}$.
Moreover, variants of topological Hochschild homology applied to log quasiregular semiperfectoid rings enjoy evenness properties analogous to those of \cite{BMS19}, and so we will make similar use of the double-speed Postnikov filtration (indicated by $\tau_{\ge 2*}$).  
\medskip

The following result is our logarithmic analog of \cite[Theorem 1.12]{BMS19}, 
which we state in full for ease of reference. The statement makes use of Breuil--Kisin twists $\widehat{\Prism}_{(A, M)}\{n\}$, see Remark \ref{rem:breuilkisin}.

\begin{theorem}\label{thm:mainthm} Let $(A, M)$ be a quasisyntomic pre-log ring.
\begin{enumerate}
\item Locally on ${\rm lqSyn}_{(A, M)}$, we have that ${\rm THH}((-, -) ; {\Bbb Z}_p), {\rm TC}^-((-, -) ; {\Bbb Z}_p)$ and ${\rm TP}((-, -) ; {\Bbb Z}_p)$ are concentrated in even degrees. 
\item The filtrations 
\begin{align*} {\rm Fil}^*{\rm THH}((A, M) ; {\Bbb Z}_p) := R\Gamma_{\rm lqsyn}((A, M), \tau_{\ge 2*}{\rm THH}((-, -) ; {\Bbb Z}_p)) \\  {\rm Fil}^*{\rm TC}^-((A, M) ; {\Bbb Z}_p) := R\Gamma_{\rm lqsyn}((A, M), \tau_{\ge 2*}{\rm TC}^-((-, -) ; {\Bbb Z}_p)) \\  {\rm Fil}^*{\rm TP}((A, M) ; {\Bbb Z}_p) := R\Gamma_{\rm lqsyn}((A, M), \tau_{\ge 2*}{\rm TP}((-, -) ; {\Bbb Z}_p))\end{align*} are complete, exhaustive, decreasing filtrations.
\item The associated graded of the filtrations in part (2) participate in natural isomorphisms \begin{align*}{\rm gr}^n{\rm THH}((A, M) ; {\Bbb Z}_p) \cong {\rm Fil}_N^{n} \widehat{\Prism}_{(A, M)}\{n\}[2n], \\ {\rm gr}^n {\rm TC}^-((A, M) ; {\Bbb Z}_p) \cong {\rm Fil}_N^{\ge n} \widehat{\Prism}_{(A, M)}\{n\}[2n], \\ {\rm gr}^n {\rm TP}((A, M) ; {\Bbb Z}_p) \cong \widehat{\Prism}_{(A, M)}\{n\}[2n].\end{align*}
\item The cyclotomic Frobenius and canonical maps \[\varphi_p, {\rm can} \colon {\rm TC}^-((A, M) ; {\Bbb Z}_p) \to {\rm TP}((A, M) ; {\Bbb Z}_p)\] descend to the associated graded, 
and hence we obtain a filtration on log topological cyclic homology ${\rm TC}((A, M) ; {\Bbb Z}_p)$. 
Writing ${\Bbb Z}_p(n)(A, M)$ for the associated graded, we have \[{\Bbb Z}_p(n)(A, M) = {\rm hofib}(\varphi_p - {\rm can} \colon {\rm Fil}^{\ge n}_N \widehat{\Prism}_{(A, M)}\{n\} \to \widehat{\Prism}_{(A, M)}\{n\}).\] 
\end{enumerate} 
\end{theorem}

\subsection{Comparison with the even filtration} In light of the theory of logarithmic ring spectra developed by Rognes \cite{Rog09} and Rognes--Sagave--Schlichtkrull \cites{RSS15, RSS18}, it is very natural to ask how the filtrations constructed here fit in the framework of Hahn--Raksit--Wilson's \emph{even filtration} \cite{HRW22}, and how variants of the even filtration may contribute to computations involving logarithmic ring spectra. We hope to pursue these questions in the future.  

\subsection{Conventions}Most of the results in this paper are stated for ordinary (commutative) rings, but it is possible to generalize them for derived or animated rings in many cases. When necessary, we will model animated rings using simplicial commutative rings in order to apply the general framework of derived logarithmic geometry as developed in \cite{SSV16} (and used in \cite{BLPO}).

Given a commutative ring $R$, we write $\GL_1(R)$ for the multiplicative group of invertible elements of $R$, while for a commutative monoid $M$ we write ${\rm GL}_1(M)$ for its submonoid of invertible elements. 

\subsection{Acknowledgments}
The authors wish to thank Teruhisa Koshikawa for several friendly conversations around the content of \cite{KY22}. 
F.B.~and T.L.~wish to thank Ben Antieau for inspiring discussions on the subject of this paper. T.L~thanks Steffen Sagave for helpful discussions relating to this material. The authors would like to thank an anonymous referee for a quick and thorough report which contained many valuable comments and corrections.  D.P.~and T.L.~were partially supported by the research training group GRK 2240 ``Algebro-Geometric methods in Algebra, Arithmetic and Topology.'' 
P.A.{\O}. acknowledges the support of the  RCN Project no. 312472 ``Equations in Motivic Homotopy Theory.''

\section{Logarithmic Hochschild homology}\label{sec:loghh} We begin by reviewing some of the results of \cite{BLPO}. In addition, we describe a useful base change property of the Gabber cotangent complex (Lemma \ref{lem:repletebasechange}), essentially due to Rognes. We also prove that the circle action on log Hochschild homology is compatible with the log de Rham differential (Proposition \ref{prop:logderham}). 

\subsection{Pre-log rings} A \emph{pre-log ring} $(R, P, \beta)$ is the datum of commutative ring $R$, a commutative monoid $P$ and a map $\beta \colon P \to (R, \cdot)$ of commutative monoids to the underlying multiplicative monoid of $R$. It is a \emph{log ring} if the canonical map $\beta^{-1}({\rm GL}_1(R)) \to {\rm GL}_1(R)$ is an isomorphism, and a pre-log ring $(R, P, \beta)$ gives rise to a log ring $(R, P^a, \beta^a)$ by defining $P^a$ to be the (non-derived) pushout of the diagram $P \xleftarrow{} \beta^{-1}({\rm GL}_1(R)) \xrightarrow{} {\rm GL}_1(R)$. We call $(R, P^a, \beta^a)$ the \emph{logification} of $(R, P, \beta)$. 

A \emph{morphism} $(f, f^\flat) \colon (R, P, \beta) \to (A, M, \alpha)$ of pre-log rings consists of a ring map $f$ and a monoid map $f^\flat$ such that the obvious diagram commutes. A monoid map $P \to M$ is \emph{strict} if $P/{\rm GL}_1(P) \to M/{\rm GL}_1(M)$ is an isomorphism \cite[Definition I.4.1.1]{Ogu18}. If $f^\flat \colon P \to M$ is strict, then the logification of the pre-log structure $M \to (A, \cdot)$ is isomorphic to that of $P \xrightarrow{f^\flat} M \xrightarrow{} (A, \cdot)$. 

Given a commutative monoid $M$, we shall denote by $\gamma \colon M \to M^{\rm gp}$ the canonical map to its group completion $M^{\rm gp}$.

\subsection{Animated pre-log rings} Let ${\rm Poly}_{(R, P)}$ denote the category of \emph{polynomial} $(R, P)$-algebras, that is, those of the form $(R[x_1, \dots, x_m, y_1, \dots, y_n], P \oplus \langle y_1, \dots, y_n \rangle)$. In analogy with the definition of animated commutative $R$-algebras (see e.g., \cite[Section 5.1.4]{Cesnav-Scholze_purity}), we define the category of \emph{animated $(R, P)$-algebras} ${\rm PreLog}_{(R, P)}^{\rm ani}$ to be that of finite product-preserving functors \[{\rm Poly}_{(R, P)}^{\rm op} \to {\rm Ani}\] to the $\infty$-category of anima. The category ${\rm PreLog}_{(R, P)}^{\rm ani}$ is the $\infty$-category underlying the simplicial model category of simplicial pre-log rings of \cite[Section 3]{SSV16}. The latter model is used in \cite{BLPO} to 
model various functors involving ${\rm PreLog}_{(R, P)}^{\rm ani}$. 

As specified in our conventions, all (relative) coproducts involving monoids and (pre-log) rings should be considered in the derived (or animated) sense. 

\subsection{Integral monoid maps are flat} A (discrete) commutative monoid $P$ is \emph{integral} if the canonical map $P \to P^{\rm gp}$ is injective, and a map $P \to M$ of integral monoids is \emph{integral} if the monoid $\pi_0(M \oplus_P N)$ is integral for any map $P \to N$ with $N$ integral. If $P \to M$ is an integral map of integral monoids, then the truncation map \begin{equation}\label{pushtrunc}M \oplus_P N \to \pi_0(M \oplus_P N)\end{equation} is an equivalence for any map of commutative monoids $P \to N$, see \cite[Proposition 4.9]{Bha12}. 

\subsection{Faithfully flat morphisms of pre-log rings} We now introduce the class of morphisms of pre-log rings that will play the role of faithfully flat morphisms of ordinary commutative rings. To motivate the following definition, let us record two observations: 

\begin{enumerate}
\item It is reasonable to impose a flatness condition on monoid maps $P \to M$ in terms of the truncation map \eqref{pushtrunc} being an equivalence: Indeed, this is always satisfied for integral maps of integral monoids. 
\item Invariants of pre-log rings tend to pass through an exactification procedure (Section \ref{subsec:repletion}), ensuring that the output lives in a category only dependent on the underlying commutative ring. For example, we consider the log cotangent complex (Section \ref{subsec:cotcx}) ${\Bbb L}_{(A, M) / (R, P)}$ an $A$-module, as opposed to an object of some hypothetical category of modules over the pre-log ring $(A, M)$. This suggests that it suffices to impose a notion of faithfulness on the underlying commutative rings: This point will become very clear in the proof of Theorem \ref{thm:gabberdescent}. 
\end{enumerate}

\noindent Based on this, we introduce the following definition:

\begin{definition}\label{def:logff} Let $(f, f^\flat) \colon (R, P) \to (A, M)$ be a map of pre-log rings. We say that
\begin{enumerate}
\item\label{eq:Bhatt-flatness}   $f^\flat$ is \emph{flat} if \eqref{pushtrunc} is an equivalence for all monoid maps $P \to N$;
\item $(f, f^\flat)$ is \emph{homologically log flat} if both $f$ and $f^\flat$ are flat; and
\item $(f, f^\flat)$ is \emph{homologically log faithfully flat} if $f$ is faithfully flat and $f^\flat$ is flat. 
\end{enumerate}

\noindent We refer to the resulting topology as the \emph{homologically log flat} or \emph{hlf} topology. Note that \eqref{eq:Bhatt-flatness} is taken from \cite[Definition 4.8]{Bha12}.
\end{definition}

\begin{remark} Our first name for the above topology was ``log fpqc''. This comes with some disadvantages: It clashes with the terminology of Molcho--Temkin \cite{MT21} and differs from the terminology of Koshikawa--Yao \cite{KY22} for the same notion. As recorded in e.g., \cite[Example 4.9]{BLPO}, there are classically log flat morphisms that fail to be homologically log flat, while the discrepancy disappears for integral maps. We, therefore, choose to borrow Koshikawa--Yao's terminology for the above notion, which they attribute to Ofer Gabber. 
\end{remark}

\subsection{The logarithmic cotangent complex}\label{subsec:cotcx} We now review the Gabber cotangent complex. One reference is Olsson \cite[Section 8]{Ols05}. The exposition here is closer in spirit to that of \cite[Section 3]{BLPO} or \cite[Section 4]{SSV16}. We refer to these sources for proof of the statements made here. 

Let $(f, f^\flat) \colon (R, P) \to (A, M)$ be a map of pre-log rings. The \emph{module of relative log differentials} $\Omega^1_{(A, M) / (R, P)}$ is the $A$-module \[\frac{\Omega^1_{A / R} \oplus \pi_0(A \otimes_{\Bbb Z} (M^{\rm gp}/P^{\rm gp}))}{{\rm d}\alpha(m) \sim \alpha(m) \otimes \gamma(m)},\] where $\alpha \colon M \to (A, \cdot)$ is the structure map and $\gamma \colon M \to M^{\rm gp}$ is the canonical map from $M$ to its group completion $M^{\rm gp}$.

This extends levelwise to a functor \[\Omega^1_{(-, -) / (R, P)} \colon {\rm sPreLog}_{(R, P)/} \to {\rm sMod}_R\] from the category of simplicial pre-log $(R, P)$-algebras to that of simplicial $R$-modules. It is left Quillen, and its total derived functor is the \emph{Gabber cotangent complex} ${\Bbb L}_{(- , -) / (R, P)}$. The associated functor of $\infty$-categories is equivalent to the left Kan extension of the functor $\Omega^1_{(-, -) / (R, P)}$ defined only on polynomial pre-log $(R, P)$-algebras $(R[x_1, \dots, x_m, y_1, \dots, y_n], \langle y_1, \dots, y_n \rangle)$ to all animated pre-log $(R, P)$-algebras. If the base pre-log ring has the trivial pre-log structure we omit the pre-log structure from the notation: That is, we set ${\Bbb L}_{(A, M) / R} := {\Bbb L}_{(A, M) / (R, \{1\})}$. 

This construction enjoys the following properties:  

\begin{enumerate}
\item For polynomial $(R, P)$-algebras \[(A, M) := (R[x_1, \dots, x_m, y_1, \dots, y_n], P \oplus \langle y_1, \dots, y_n \rangle)\] there is a canonical identification \[{\Bbb L}_{(A, M) / (R, P)} = \Omega^1_{(A, M) / (R, P)} = \bigoplus_{i = 1}^m A\{{\rm d}x\} \oplus \bigoplus_{i = 1}^n A\{{\rm dlog}(y)\},\] and the resulting map $\Omega^1_{A/R} \to \Omega^1_{(A, M) / (R, P)}$ sends generators of the form ${\rm d}y$ to $y{\rm dlog}(y)$.
\item If $f^\flat$ is strict, then the canonical map ${\Bbb L}_{A / R} \to {\Bbb L}_{(A, M) / (R, P)}$ is an equivalence.
\item\label{eq:base change} If $(C, K)$ arises as the homotopy pushout of \[(B, N) \xleftarrow{(g, g^\flat)} (R, P) \xrightarrow{(f, f^\flat)} (A, M),\] then the canonical map $C \otimes_B {\Bbb L}_{(B, N) / (R, P)} \to {\Bbb L}_{(C, K) / (A, M)}$ is an equivalence. 
\item \label{eq:logcotangent_vsnonlog} There is a canonical equivalence $\L_{(\Z[M],M)/ (\Z[P],P)} \cong \Z[M]\otimes_\Z (M^{\gp}/P^{\gp})$, and a cofiber sequence of $A$-modules
\[
A \otimes_\Z (M^{\gp}/P^{\gp}) \to \L_{(A,M)/(R,P)} \to \L_{A/ R\otimes_{\Z[P]} \Z[M]}
\]
\item\label{eq:transitivity} For a composite $(R, P) \xrightarrow{} (A, M) \xrightarrow{} (B, N)$, there is a cofiber sequence \[B \otimes_A {\Bbb L}_{(A, M) / (R, P)} \to {\Bbb L}_{(B, N) / (R, P)} \to {\Bbb L}_{(B, N) / (A, M)}\] of $B$-modules.  
\item If $(R, P) \to (A, M)$ is an \emph{integral} log smooth morphism, then the truncation map 
\[{\Bbb L}_{(A, M) / (R, P)} \to \Omega^1_{(A, M) / (R, P)}\] is an equivalence. See \cite[Proposition 4.6]{BLPO} for a precise relationship between  the notion of log smoothness in classical log geometry and its derived counterpart. 
\end{enumerate}

\subsection{Homologically log flat descent for the cotangent complex} By \cite[Theorem 3.1]{BMS19}, the wedge powers of the (non-logarithmic) cotangent complex satisfy flat descent. Using the HKR theorem, this is   the key input in proving flat descent for topological Hochschild homology. 
For the analogous statement in the log setting, we use the notion of homologically log faithfully flat of Definition \ref{def:logff}:

\begin{theorem}\label{thm:gabberdescent} Consider a base pre-log ring $(R, P)$, and let $(A, M) \to (B, N)$ be a homologically log faithfully flat map of $(R, P)$-algebras. Then \[\wedge^i {\Bbb L}_{(A, M)/(R, P)} \simeq {\rm lim}(\wedge^i {\Bbb L}_{(B, N) / (R, P)} \rightrightarrows \wedge^i {\Bbb L}_{(B \otimes_A B, N \oplus_M N) / (R, P)} \cdots).\] 
\end{theorem}

\begin{proof} We follow the steps of \cite[Theorem 3.1]{BMS19}, and proceed by induction on $i$, where the case   $i = 0$ is just faithfully flat descent for rings. The argument for $i = 1$ will, as spelled out below, be as in \cite[Theorem 3.1]{BMS19}, and this implies the result for higher $i$ by the filtrations considered in \cite[Section V.4]{Ill71}.  

Let $(B^{\bullet}, N^{\bullet})$ denote the Cech nerve of the map $(A, M) \to (B, N)$. We can apply the transitivity sequence of \ref{subsec:cotcx}.\eqref{eq:transitivity} to the composite $(R, P) \to (A, M) \to (B^{\bullet}, N^{\bullet})$ to obtain a cosimplicial cofiber sequence \[B^{\bullet} \otimes_A {\Bbb L}_{(A, M)/(R, P)} \to {\Bbb L}_{(B^{\bullet}, N^{\bullet}) / (R, P)} \to {\Bbb L}_{(B^{\bullet}, N^{\bullet}) / (A, M)}.\] Since $A \to B$ is faithfully flat, fpqc descent implies that the induced map 
\[{\Bbb L}_{(A, M)/(R, P)} \to {\rm lim}(B^{\bullet} \otimes_A {\Bbb L}_{(A, M)/(R, P)})\] is an isomorphism. Again by the transitivity sequence, it  suffices to prove that the totalization of ${\Bbb L}_{(B^{\bullet}, N^{\bullet}) / (A, M)}$ vanishes. For this, we show that each $\pi_i {\Bbb L}_{(B^{\bullet}, N^{\bullet}) / (A, M)}$ vanishes (the cosimplicial direction of $(B^{\bullet}, N^{\bullet})$ is still running). Since $A \to B$ is faithfully flat, it suffices to prove that (the right-hand side of) \begin{equation}\label{descreduction1}B \otimes_A \pi_i {\Bbb L}_{(B^{\bullet}, N^{\bullet}) / (A, M)} \simeq \pi_i (B \otimes_A {\Bbb L}_{(B^{\bullet}, N^{\bullet}) / (A, M)})\end{equation} vanishes. Descent for $B$ together with the base change property \ref{subsec:cotcx}.\eqref{eq:base change} of the Gabber cotangent complex (using again that $A \to B$ is faithfully flat) implies that \[B \otimes_A {\Bbb L}_{(B^{\bullet}, N^{\bullet}) / (A, M)} \xrightarrow{\simeq} (B \otimes_A B^{\bullet}) \otimes_{B^{\bullet}} {\Bbb L}_{(B^{\bullet}, N^{\bullet}) / (A, M)} \xrightarrow{\simeq} {\Bbb L}_{(B \otimes_A B^{\bullet}, N \oplus_M N^{\bullet}) / (B, N)}.\] But now $(B, N) \to (B \otimes_A B^{\bullet}, N \oplus_M N^{\bullet})$ is the Cech nerve of a map with a section, hence a cosimplicial homotopy equivalence. This shows that the right-hand side of \eqref{descreduction1} is equivalent to $\pi_i {\Bbb L}_{(B, N)/(B, N)} \simeq 0$, which concludes the proof. 
\end{proof}

\subsection{Repletion}\label{subsec:repletion} The key ingredient to the definition of log Hochschild homology pursued by Rognes is the notion of \emph{repletion}, or \emph{exactification}, of a map of monoids. This operation was introduced by Kato in \cite{Kat89} in the geometric context and was revisited and generalized to the context of topological log structures by Rognes \cite{Rog09}. We shall use Rognes' terminology throughout.   

\begin{definition} Let $(C, K) \to (A, M)$ be a map of pre-log rings with the property that $K^{\rm gp} \to M^{\rm gp}$ is surjective.
\begin{enumerate}
\item The \emph{repletion} $K^{\rm rep} \to M$ is the pullback of $K^{\rm gp} \to M^{\rm gp}$ along $M \to M^{\rm gp}$. 
\item The \emph{replete base change} $C^{\rm rep}$ is the base change $C \otimes_{{\Bbb Z}[K]} {\Bbb Z}[K^{\rm rep}]$. 
\end{enumerate}
\end{definition}

Note that, in general, $C \otimes_{{\Bbb Z}[K]} {\Bbb Z}[K^{\rm rep}]$ may not be discrete. By definition, the resulting map $K^{\rm rep} \to M$ is exact. Since the monoids involved are integral, it is strict by 
\cite[Proposition 4.2.1(5)]{Ogu18}. 
\begin{remark}
       Note that the exactification has a very concrete geometric interpretation in terms of affine blow-ups of the corresponding toric variety. For closed embeddings, the exactification procedure is the passage to a blow-up,  after which the morphism becomes strict. 
     We view the passage from $(C, K)$ to $(C^{\rm rep}, K^{\rm rep})$ as a means of ``cashing out'' the pre-log structure on $C$ (relative to that on $M$), allowing us to work with the resulting commutative ring $C^{\rm rep}$ while still retaining information from the log structure. If $M$ is trivial, this corresponds precisely to passing to the open complement of the log structure. The following example supports this philosophy and  motivates the definition of log Hochschild homology pursued in \cite{BLPO}.
    \end{remark}

\begin{example}\label{ex:derivedlogdiagonal} Let $(R, P) \to (A, M)$ be a map of pre-log rings. There is a factorization \begin{equation}\label{diagonal}\pi_0(A \otimes_R A, M \oplus_P M) \to \pi_0((A \otimes_R A)^{\rm rep}, (M \oplus_P M)^{\rm rep}) \to (A, M)\end{equation} of the (underived) diagonal $\pi_0(A \otimes_R A, M \oplus_P M) \to (A, M)$. The indecomposables (i.e., conormal) of the underlying map $\pi_0(A \otimes_R A)^{\rm rep} \to A$ is the module $\Omega^1_{(A, M) / (R, P)}$ of relative log differentials \cite[Proposition 4.2.8(ii)]{KS04}. In the derived setting, we instead obtain the Gabber cotangent complex ${\Bbb L}_{(A, M) / (R, P)}$ as the (derived) conormal of the log diagonal \[(A \otimes_R A, M \oplus_P M) \to ((A \otimes_R A)^{\rm rep}, (M \oplus_P M)^{\rm rep}) \to (A, M);\] this is the content of \cite[Proposition 3.17]{BLPO}. Note that, geometrically, the diagonal embedding $X = \mathrm{Spec}(A) \hookrightarrow \mathrm{Spec}(A\otimes_R A) = X\times_S X$ is not an exact (equivalently, strict) closed immersion, where $X$ is seen as a log scheme with log structure induced by $M$, and $X\times_S X$ is a log scheme with log structure induced by $\pi_0(M\oplus_P M)$. One achieves strictness precisely by passing to the replete base change that models an affine chart for a suitable blow-up. See \cite[2.16]{HK}. \end{example}

\subsection{Replete base change for the Gabber cotangent complex} A key ingredient for the proof of Theorem \ref{thm:logthhdefhh} is the following base change property for the log cotangent complex. From a log geometric perspective, this  expresses the fact that the morphism $(A, M) \to (A^{\rm rep}, M^{\rm rep})$ is (formally derived) log \'etale. 

\begin{lemma}\label{lem:repletebasechange} Let $(R, P) \to (A, M)$ be a map of pre-log rings. Assume that there is a map $(A, M) \to (S, Q)$ with the property that the induced map $M^{\rm gp} \to Q^{\rm gp}$ is a surjection, and let $(A^{\rm rep}, M^{\rm rep}) \to (S, Q)$ denote the repletion of $(A, M) \to (S, Q)$.

Then the canonical map \[A^{\rm rep} \otimes_A {\Bbb L}_{(A, M)/(R, P)} \xrightarrow{} {\Bbb L}_{(A^{\rm rep}, M^{\rm rep}) / (R, P)}\] is an equivalence. 
\end{lemma}

\begin{proof} We follow a proof strategy in \cite[Proposition 11.30]{Rog09} in the context of log ${\rm TAQ}$. 
Recall that, using \ref{subsec:cotcx}.\eqref{eq:logcotangent_vsnonlog}, we have pushout squares of $A$-modules
\begin{equation}\label{eq:repletebc_eq}
\begin{tikzcd}[column sep = small]
A\otimes_{\Z[M]} \L_{\Z[M]/\Z[P]} \arrow[r] \arrow[d] & A \otimes_\Z (M^\gp /P^\gp)\arrow[d] \\
\L_{A/R} \arrow[r] & \L_{(A,M)/(R,P)}
\end{tikzcd}\quad 
\begin{tikzcd}[column sep = small]
A^{\rm rep}\otimes_{\Z[M^{\rm rep}]} \L_{\Z[M^{\rm rep}]/\Z[P]} \arrow[r] \arrow[d] & A \otimes_\Z ({(M^{\rm rep})}^\gp /P^\gp)\arrow[d] \\
\L_{A^{\rm rep}/R} \arrow[r] & \L_{(A^{\rm rep},M^{\rm rep})/(R,P)},
\end{tikzcd}  
\end{equation}
see \cite[(3.3)]{BLPO}, and a canonical map from the base change to $A^{\rm rep}$ of the left square of \eqref{eq:repletebc_eq} to the right square. The transitivity sequence \ref{subsec:cotcx}.\eqref{eq:transitivity} yields then a commutative diagram
\[\begin{tikzcd}A^{\rm rep} \otimes_A {\Bbb L}_{A/R} \ar{d} & A^{\rm rep} \otimes_{{\Bbb Z}[M]} {\Bbb L}_{{\Bbb Z}[M]/{\Bbb Z}[P]} \ar{r} \ar{l} \ar{d} & A^{\rm rep} \otimes_{\Bbb Z} (M^{\rm gp}/P^{\rm gp}) \ar{d}{\simeq} \\ {\Bbb L}_{A^{\rm rep} / R} \ar{d} & A^{\rm rep} \otimes_{{\Bbb Z}[M^{\rm rep}]} {\Bbb L}_{{\Bbb Z}[M^{\rm rep}] / {\Bbb Z}[P]} \ar{l} \ar{r} \ar{d} & A^{\rm rep} \otimes_{\Bbb Z} ((M^{\rm rep})^{\rm gp}/P^{\rm gp}) \ar{d} \\ {\Bbb L}_{A^{\rm rep}/A} & A^{\rm rep} \otimes_{{\Bbb Z}[M^{\rm rep}]} {\Bbb L}_{{\Bbb Z}[M^{\rm rep}] / {\Bbb Z}[M]} \ar[swap]{l}{\simeq} \ar{r} & *\end{tikzcd}\] 
of $A^{\rm rep}$-modules, where the vertical columns are cofiber sequences. 

Base change for the cotangent complex implies that the map $A^{\rm rep} \otimes_{{\Bbb Z}[M^{\rm rep}]} {\Bbb L}_{{\Bbb Z}[M^{\rm rep}] / {\Bbb Z}[M]} \to {\Bbb L}_{A^{\rm rep} / A}$ is an equivalence so that the upper left-hand square is a pushout. Since $M^{\rm gp} \to (M^{\rm rep})^{\rm gp}$ is an equivalence, the same holds for the map $A^{\rm rep} \otimes (M^{\rm gp}/P^{\rm gp}) \to A^{\rm rep} \otimes ((M^{\rm rep})^{\rm gp}/P^{\rm gp})$.   Thus the canonical map from the pushout of the upper horizontal row to the pushout of the middle horizontal row is an equivalence, and the result follows.
\end{proof}

\subsection{Logarithmic Hochschild homology}  Let $R \to A$ be a map of commutative rings. The \emph{Hochschild homology} ${\rm HH}(A/R)$ of $A$ relative to $R$ can be defined as the tensor $S^1 \otimes_R A$ in simplicial commutative $R$-algebras. Writing $S^1$ as the homotopy pushout of $* \xleftarrow{} * \sqcup * \xrightarrow{} *$, we find that ${\rm HH}(A/R) = A \otimes_{A \otimes_R A} A$. 

Motivated by Example \ref{ex:derivedlogdiagonal}, we gave in \cite[Definition 5.3]{BLPO} the following:

\begin{definition}\label{def:loghh} Let $(R, P) \to (A, M)$ be a map of pre-log rings. The \emph{logarithmic Hochschild homology} of $(A, M)$ relative to $(R, P)$ is the pushout \[{\rm HH}((A, M)/(R, P)) := A \otimes_{(A \otimes_R A)^{\rm rep}} A\] of simplicial commutative $A$-algebras. 
\end{definition}

\begin{remark} In algebro-geometric terms, the description of Hochschild homology as the iterated coproduct $A \otimes_{A \otimes_R A} A$ exhibits ${\rm HH}(A/R)$ as the (derived) self-intersections of the diagonal $A \otimes_R A \to A$. In the language of Kato--Saito \cite[Section 4]{KS04}, Definition \ref{def:loghh} exhibits ${\rm HH}((A, M)/(R, P))$ as the (derived) self-intersections of the (derived) log diagonal $(A \otimes_R A)^{\rm rep} \to A$, see Example \ref{ex:derivedlogdiagonal}. Note that in \cite{BLPO}, we used the  notation $\log\HH$ instead of $\HH$ to denote logarithmic Hochschild homology.
\end{remark}

\begin{remark}\label{rmk:properties_logHH}We record some key properties of log Hochschild homology here:

\begin{enumerate}

\item In analogy with the transitivity formula ${\rm HH}(B / A) \simeq A \otimes_{{\rm HH}(A / R)} {\rm HH}(B / R)$ for a composite $R \to A \to B$, there is an equivalence \[{\rm HH}((B, N) / (A, M)) \simeq A \otimes_{{\rm HH}((A, M) / (R, P))} {\rm HH}((B, N) / (R, P))\] for a composite $(R, P) \to (A, M) \to (B, N)$ of pre-log rings. This is the content of \cite[Proposition 5.4]{BLPO}.
\item\label{eq:logHH_as_Rognes} The formulation of Definition \ref{def:loghh} coincides with that pursued by Rognes \cite{Rog09}. In particular, for a map of commutative monoids $P \to M$, we can form the simplicial tensor $S^1 \oplus_P M$, and the repletion $S^1 \oplus_P^{\rm rep} M$ of the augmentation $S^1 \oplus_P M \to M$. Then ${\rm HH}((A, M) / (R, P))$ can equivalently be described as the pushout of the diagram \begin{equation}\label{rognespushout}{\rm HH}(A/R) \xleftarrow{} A \otimes_{{\Bbb Z}[M]} {\Bbb Z}[S^1 \oplus_P M] \xrightarrow{} A \otimes_{{\Bbb Z}[M]} {\Bbb Z}[S^1 \oplus_P^{\rm rep} M]\end{equation} of derived $A$-algebras; this is the content of \cite[Proposition 1.4]{BLPO}. 
\item Log Hochschild homology is invariant under passing to the associated log ring:
There are equivalences \[{\rm HH}((A, M) / (R, P)) \xrightarrow{\simeq} {\rm HH}((A, M^a) / (R, P)) \xrightarrow{\simeq} {\rm HH}((A, M^a) / (R, P^a));\] this can be deduced from Theorem \ref{thm:loghkrfilt} below and the analogous statement for the Gabber cotangent complex \cite[Theorem 8.16]{Ols05}. In particular, this implies that log Hochschild recovers classical Hochschild homology for strict morphisms, as in this case, the canonical map \[{\rm HH}(A/R) \xrightarrow{\simeq} {\rm HH}((A, P) / (R, P)) \xrightarrow{\simeq} {\rm HH}((A, M) / (R, P))\] is an equivalence by logification invariance. This implies that there is an equivalence ${\rm HH}((A, M) / (R, P)) \xrightarrow{\simeq} {\rm HH}((A, N) / (R, P))$ for strict maps $M \to N$.  
\end{enumerate}
\end{remark}

\subsection{The logarithmic ${\rm HKR}$-filtration} Recall  that Hochschild homology ${\rm HH}(A/R)$ admits a descending separated filtration with graded pieces $(\wedge_A^i {\Bbb L}_{A / R})[i]$ (see e.g., \cite[Proposition IV.4.1]{NS18}). The following is the log analog of that result:

\begin{theorem}\label{thm:loghkrfilt}\cite[Theorem 1.1]{BLPO} Let $(R, P) \to (A, M)$ be a map of pre-log rings. Then: \begin{enumerate}
    \item  Log Hochschild homology ${\rm HH}((A, M) / (R, P))$ admits a separated, descending filtration with graded pieces $(\wedge_A^i {\Bbb L}_{(A, M) / (R, P)})[i]$.  
\item
If the morphism $(R, P) \to (A, M)$ is \emph{derived} log smooth,   the canonical map \begin{equation}\label{eq:HKR_lsm}
\Phi_* \colon \Omega^*_{(A, M) / (R, P)} \to \pi_*{\rm HH}((A, M) / (R, P))
\end{equation}
is an isomorphism of strictly commutative graded rings. 
\end{enumerate}
\end{theorem}
Note that \eqref{eq:HKR_lsm} is always an isomorphism in degrees $0$ and $1$ by \cite[Proposition 5.15]{BLPO}. 
We shall refer to the above filtration as the \emph{log ${\rm HKR}$-filtration}. 

\subsection{The circle action on logarithmic Hochschild homology} Let $P \to M$ be a map of commutative monoids. The linearization of the replete bar construction \[{\rm HH}(({\Bbb Z}[M], M) / ({\Bbb Z}[P], P)) \simeq {\Bbb Z}[M \times_{M^{\rm gp}} B^{\rm cy}_{P^{\rm gp}}(M^{\rm gp})_{\bullet}]\] is a simplicial commutative ring. Rognes' formulation (the right-hand side of the displayed equivalence) is more convenient 
for discussing the circle action on logarithmic Hochschild homology. 

The operator \[t_q(m, (g_0, g_1, \cdots, g_q)) = (-1)^q(m, (g_q, g_0, \cdots  ,g_{q - 1}))\] satisfies the formulas of \cite[Definition 2.5.1]{Lod92}. In particular, the replete bar construction has the structure of a mixed complex, and there is a degree-increasing differential \[B \colon \pi_q\HH(({\Bbb Z}[M], M) / ({\Bbb Z}[P], P)) \to \pi_{q + 1}\HH(({\Bbb Z}[M], M) / ({\Bbb Z}[P], P)).\] Since the map ${\rm HH}({\Bbb Z}[M] / {\Bbb Z}[P]) \to {\rm HH}(({\Bbb Z}[M], M) / ({\Bbb Z}[P], P))$ is compatible with the cyclic operators, we obtain a degree-increasing differential \begin{equation}\label{logcomparison}B \colon \pi_*{\rm HH}((A, M) / (R, P)) \to \pi_{* + 1}((A, M) / (R, P))\end{equation} on log Hochschild homology $\pi_*{\rm HH}((A, M) / (R, P))$.

Log Hochschild homology is a cyclic $R$-module. We define \emph{log cyclic, negative and periodic} homology as the homotopy orbits, homotopy fixed points, and the Tate construction of the resulting $S^1$-action: \begin{align*}{\rm HC}((A, M) / (R, P)) := {\rm HH}((A, M) / (R, P))_{hS^1}, \\ {\rm HC}^-((A, M) / (R, P)) := {\rm HH}((A, M) / (R, P))^{hS^1}, \\ {\rm HP}((A, M) / (R, P)) := {\rm HH}((A, M) / (R, P))^{tS^1},\end{align*} cf.\ \cite[Theorem 2.1]{Hoy18} for comparison with classical definitions of these variants of Hochschild homology.

The maps \eqref{logcomparison} are compatible with the log de Rham differential:

\begin{proposition}\label{prop:logderham} The diagram \[\begin{tikzcd}\Omega^q_{(A, M)/(R, P)} \ar{r}{{\rm d}} \ar{d}{\Phi_q} &  \Omega^{q + 1}_{(A, M)/(R, P)} \ar{d}{\Phi_{q + 1}} \\ \pi_q\HH((A, M) / (R, P))  \ar{r}{B} & \pi_{q + 1}\HH((A, M) / (R, P))\end{tikzcd}\] is commutative. 
\end{proposition}

\begin{proof} By the analogous statement for classical Hochschild homology \cite[Proposition 2.3.3]{Lod92}, we find by \cite[Proposition V.2.1.1]{Ogu18} that it suffices to prove that $B(\Phi_1({\rm dlog}(m)))$ is zero in $\pi_2{\rm HH}((A, M) / (R, P))$.

The computation of \cite[Proposition 5.15]{BLPO} shows that the defining pushout \eqref{rognespushout} of ${\rm HH}((A, M) / (R, P))$ gives rise to the pushout \[\begin{tikzcd}A \otimes_{{\Bbb Z}[M]} \Omega^1_{{\Bbb Z}[M] / {\Bbb Z}[P]} \ar{r} \ar{d} & A \otimes_{{\Bbb Z}[M]} \Omega^1_{({\Bbb Z}[M], M) / ({\Bbb Z}[P], P)} \ar{d} \\ \Omega^1_{A / R} \ar{r} & \Omega^1_{(A, M) / (R, P)}\end{tikzcd}\] of $A$-modules on $\pi_1$. Since we are only concerned with the image of a logarithmic differential ${\rm dlog}(m) \in \Omega^1_{({\Bbb Z}[M], M) / ({\Bbb Z}[P], P)}$ under the comparison map $\Phi_1$, we reduce to checking that $B(\Phi_1({\rm dlog}(m)))$ is zero in $\pi_2{\rm HH}(({\Bbb Z}[M], M) / ({\Bbb Z}[P], P))$. 

 The element $\Phi_1({\rm dlog}(m))$ is represented by the $1$-cycle $(1, (\gamma(m)^{-1}, \gamma(m)))$ in the Moore complex of ${\Bbb Z}[B^{\rep}_P(M)] \simeq {\rm HH}(({\Bbb Z}[M], M) / ({\Bbb Z}[P], P))$. Spelling out the definition of Connes' $B$-operator in this case (see \cite[Page 57]{Lod92}), we obtain \begin{align*}B(1, (\gamma(m)^{-1}, \gamma(m))) &=  (1, (1, \gamma(m)^{-1}, \gamma(m))) - (1, (1, \gamma(m), \gamma(m)^{-1})) \\ & + (1, (\gamma(m)^{-1}, 1, \gamma(m))) - (1, (\gamma(m) , 1, \gamma(m)^{-1})).\end{align*} Under the simplicial isomorphism ${\Bbb Z}[B^{\rep}_P(M)] \cong {\Bbb Z}[M] \otimes_{R} R[B(M^{\gp}/P^{\gp})]$ (see e.g., \cite[Lemma 5.10]{BLPO}), this element is sent to \begin{align*}& (1 \otimes \gamma(m)^{-1} \otimes \gamma(m)) - (1 \otimes \gamma(m) \otimes \gamma(m)^{-1})  +  (1 \otimes 1 \otimes \gamma(m)) - (1 \otimes 1 \otimes \gamma(m)^{-1}).\end{align*} Keeping in mind that the first coordinate of ${\Bbb Z}[M] \otimes_{{\Bbb Z}} R[B(M^{\gp}/P^{\gp})]$ does not interact with the simplicial structure, the remaining task is to exhibit \[(\gamma(m)^{-1} \otimes \gamma(m)) - (\gamma(m) \otimes \gamma(m)^{-1}) + (1 \otimes \gamma(m)) - (1 \otimes \gamma(m)^{-1})\] as a boundary in the Moore complex of ${\Bbb Z}[B(M^{\gp}/P^{\gp})]$. Direct computation shows that it is hit by \[
\gamma(m)\otimes \gamma(m)^{-1}\otimes \gamma(m) + 2(1\otimes 1 \otimes \gamma(m)) + \gamma(m)\otimes 1\otimes 1  -1\otimes 1\otimes \gamma(m)^{-1}
,\] which concludes the proof.
\end{proof}

By the log ${\rm HKR}$-filtration (Theorem \ref{thm:loghkrfilt}) and flat descent for the Gabber cotangent complex (Theorem \ref{thm:gabberdescent}), the following is now a consequence of the argument in \cite[Corollary 3.4]{BMS19}:

\begin{proposition}\label{prop:loghhflatdescent} Let $(R, P)$ be a pre-log ring. Then the functors \begin{align*}{\rm HH}((-, -) / (R, P)), \quad {\rm HC}^-((-, -) / (R, P)), \\ {\rm HH}((-, -) / (R, P))_{hS^1}, \quad {\rm HP}((-, -) / (R, P))\end{align*} are  hlf sheaves. \qed
\end{proposition}

Similarly, replete base change for the Gabber cotangent complex (Lemma \ref{lem:repletebasechange}) implies:

\begin{proposition}\label{prop:loghhrepletebasechange} Let $(R, P) \to (A, M)$ be a map of pre-log rings. Assume that there is a map $(A, M) \to (S, Q)$ with the property that the induced map $M^{\rm gp} \to Q^{\rm gp}$ is a surjection, and let $(A^{\rm rep}, M^{\rm rep}) \to (S, Q)$ denote the repletion of $(A, M) \to (S, Q)$. Then the canonical map \begin{equation}\label{repbasechange}A^{\rm rep} \otimes_A {\rm HH}((A, M) / (R, P)) \xrightarrow{} {\rm HH}((A^{\rm rep}, M^{\rm rep}) / (R, P)) \end{equation} is an equivalence.

\end{proposition}

\begin{proof} By Theorem \ref{thm:loghkrfilt}, the map \eqref{repbasechange} is a map of abutments of strongly convergent spectral sequences with $E^2$-terms \[ \wedge_{A^{\rm rep}}^* (A^{\rm rep} \otimes_A {\Bbb L}_{(A, M) / (R, P)}) \to \wedge^*_{A^{\rm rep}} {\Bbb L}_{(A^{\rm rep}, M^{\rm rep}) / (R, P)},\] from which the result follows from Lemma \ref{lem:repletebasechange}.
\end{proof}

As sketched in the introduction, we have established all necessary ingredients to apply the techniques of \cite{Ant19} to prove Theorem \ref{thm:logantieau}. We summarize this here:

\begin{proof}[Proof of Theorem \ref{thm:logantieau}] As explained in \cite[Page 510]{Ant19}, combining Example 2.4 of \emph{loc.\ cit.} with the log ${\rm HKR}$-filtration of Theorem \ref{thm:loghkrfilt} gives an ${\Bbb N}$-indexed filtration on log negative cyclic homology ${\rm HC}^-((A, M) / (R, P))$ with graded pieces cochain complexes of the form \[ 0 \to \Omega^n_{(A, M) / (R, P)} \to \Omega^{n + 1}_{(A, M) / (R, P)} \to \cdots.\] By Proposition \ref{prop:logderham}, the differential is the log de Rham differential. Inspecting \cite{Ant19}, this suffices to establish Theorem \ref{thm:logantieau}.
\end{proof}

\section{Logarithmic topological Hochschild homology}\label{sec:logthh} We now review the notion of logarithmic \emph{topological} Hochschild homology, as first introduced by Rognes \cite{Rog09}. As we shall only be concerned with pre-log ring spectra whose underlying ``commutative monoid'' is an ordinary commutative monoid, the theory simplifies significantly. In particular, we do not use the machinery of \cites{SS12, Sag14, RSS15, RSS18} to handle ``graded ${\Bbb E}_{\infty}$-spaces'' to capture homotopy classes in non-zero degrees. 

\subsection{Definition and first properties} For our purposes, a \emph{pre-log ring spectrum} $(A, M)$ consists of an ${\Bbb E}_{\infty}$-ring $A$, a commutative monoid $M$, and a map ${\Bbb S}[M] \to A$ of ${\Bbb E}_{\infty}$-rings.

The following is a variant of the definition of log topological Hochschild homology pursued in the second-named author's thesis \cites{Lun21, Lun22}, and is also closely related to \cite[Section 13]{Rog09}. 

\begin{definition} Let $(R, P) \to (A, M)$ be a map of pre-log ring spectra. The \emph{log topological Hochschild homology} ${\rm THH}((A, M) / (R, P))$ is defined as the pushout of the diagram \[A \xleftarrow{} (A \otimes_R A) \otimes_{{\Bbb S}[M \oplus_P M]} {\Bbb S}[(M \oplus_P M)^{\rm rep}] \xrightarrow{} A\] of ${\Bbb E}_{\infty}$-rings. 
\end{definition}

If $(R, P) = ({\Bbb S}, \{1\})$ with structure map the identity, 
we omit it from the notation and simply write $\THH(A,M)$.

\begin{remark}
    The definition is motivated by its relationship with the (spectral) log cotangent complex in \cite{Lun21}, in analogy with Example \ref{ex:derivedlogdiagonal}.
Note that if $(R, P) \to (A, M)$ is a map of ordinary pre-log rings, this recovers the log Hochschild homology ${\rm HH}((A, M)/(R, P))$ introduced in Definition \ref{def:loghh} (it is enough to spell out the replete base change). 
\end{remark}
As in the case of ordinary topological Hochschild homology, we have the following transitivity property:

\begin{proposition}{\cite[Lemma 5.4]{Lun21}}\label{lem:transitivitylogthh} Let $(R, P) \to (A, M) \to (B, N)$ be maps of pre-log ring spectra. There is a natural equivalence \[{\rm THH}((B, N) / (A, M)) \simeq A \otimes_{{\rm THH}((A, M) / (R, P))} {\rm THH}((B, N) / (R, P))\] of ${\Bbb E}_{\infty}$-rings.  
\end{proposition}

We shall mostly use the following consequence:

\begin{corollary}\label{cor:transitivitylogthh} Let $(A, M)$ be an ordinary pre-log ring. There is a natural equivalence \[{\rm HH}(A, M) \simeq {\Bbb Z} \otimes_{{\rm THH}({\Bbb Z})} {\rm THH}(A, M)\] of ${\Bbb E}_{\infty}$-rings.  
\end{corollary}

In particular, there is a canonical map \begin{equation}\label{logthhlinearization}{\rm THH}(A, M) \to {\rm HH}(A, M)\end{equation} of ${\Bbb E}_{\infty}$-rings, obtained by extension of scalars along the augmentation ${\rm THH}(\Bbb Z) \to {\Bbb Z}$. 

\begin{lemma}\label{lem:logthhlinearization} The map \eqref{logthhlinearization} induces an isomorphism \[\pi_i{\rm THH}(A, M) \to \pi_i{\rm HH}(A, M)\] for $i \le 2$.  
\end{lemma}

\begin{proof}  This follows B\"okstedt's computation of $\pi_*{\rm THH}({\Bbb Z})$ \cite{Bo86} (in particular, its vanishing in degrees $1$ and $2$) and the ${\rm Tor}$-spectral sequence \[E^2_{p,q} = {\rm Tor}_p^{\pi_*{\rm THH}({\Bbb Z})}({\Bbb Z}, \pi_*{\rm THH}(A, M))_q \implies \pi_*{\rm HH}(A, M)\] obtained from Corollary \ref{cor:transitivitylogthh}.  
\end{proof}

Finally, we will need to refer to Rognes' formulation of log topological Hochschild homology:

\begin{proposition}\label{lem:rognescomp} There is a natural equivalence \[{\rm THH}(A, M) \simeq {\rm THH}(A) \otimes_{{\Bbb S}[B^{\rm cyc}(M)]} {\Bbb S}[M \times_{M^{\rm gp}} B^{\rm cyc}(M^{\rm gp})]\] of ${\Bbb E}_{\infty}$-ring spectra. 
\end{proposition}

\begin{proof} This can be deduced from \cite[Section 13]{Rog09}, or by applying \cite[Proposition 5.8]{Lun21} to the class of pre-log ring spectra considered here. 
\end{proof}

\subsection{Cyclotomic structure}\label{subsec:cyclotomic} 
The description of Proposition \ref{lem:rognescomp} gives a convenient way to characterize the circle action on log topological Hochschild homology. As explained below, this gives a cyclotomic structure on absolute log topological Hochschild homology. 

\begin{construction}\label{constr:cycl_structure} Let $M$ be a (discrete) commutative monoid. Recall from \cite[Lemma IV.3.1]{NS18} that topological Hochschild homology of the spherical monoid ring ${\Bbb S}[M]$ is identified with the suspension spectrum ${\Bbb S}[B^{\rm cyc}(M)] := \Sigma^\infty_+B^{\rm cyc}(M)$  of the geometric realization   of the cyclic bar construction $B^{\rm cyc}_\bullet(M)$.  

By  \cite[Proposition B.22]{NS18}, the space $B^{\rm cyc}(M)$ admits a canonical ${\Bbb T}$-action, and by \cite[Lemma IV.3.1(i)]{NS18} there is, for every prime $p$, a ${\Bbb T}$-equivariant map 
\[\psi_p\colon B^{\rm cyc}(M) \to B^{\rm cyc}(M)^{hC_p} \]
which refines the cyclotomic structure on  ${\rm THH}({\Bbb S}[M])$, in the sense that 
it fits in a commutative diagram 
\[\begin{tikzcd}{\rm THH}({\Bbb S}[M]) \cong {\Bbb S}[B^{\rm cyc}(M)] \ar{d}{\varphi_p} \ar{r}{{\Bbb S}[\psi_p]} & {\Bbb S}[B^{\rm cyc}(M)^{hC_p}] \ar{d} \\ {\rm THH}({\Bbb S}[M])^{tC_p} \cong {\Bbb S}[B^{\rm cyc}(M)]^{tC_p} & {\rm THH}({\Bbb S}[M])^{hC_p} \cong {\Bbb S}[B^{\rm cyc}(M)]^{hC_p} \ar{l}{{\rm can}}.\end{tikzcd}\] 
where $\varphi_p$ is the cyclotomic structure map of ${\rm THH}({\Bbb S}[M])$. 
By construction \cite[Lemma IV.3.1(ii)]{NS18}, the map $\psi_p$ sits in a commutative diagram \[\begin{tikzcd}M \ar{r} \ar{d}{\Delta} & B^{\rm cyc}(M) \ar{d}{\psi_p} \\ (M \times \cdots \times M)^{hC_p} \ar{r} & B^{\rm cyc}(M)^{hC_p}.\end{tikzcd}\] 
where the top horizontal arrow is the map induced by the inclusion of $M$ as $0$-simplex of $B^{\rm cyc}_\bullet(M)$ to its colimit $B^{\rm cyc}(M)$, and similarly the lower horizontal arrow is induced from the inclusion of the $[p]$-th object into the colimit. In particular, if we replace $M$ by  $M^{\rm gp}$, we obtain that $\psi_p$ is compatible with the augmentation map $B^{\rm cyc}(M^{\rm gp}) \to M^{\rm gp}$, in the sense that the diagram
\[\begin{tikzcd}B^{\rm cyc}(M^{\rm gp}) \ar{dd}{\psi_p} \ar{r} & M^{\rm gp} \ar{d}{\Delta} & M \ar{d}{\Delta} \ar{l} \\ \vspace{10 mm} & (M^{\rm gp} \times \cdots \times M^{\rm gp})^{hC_p} \ar{d} & (M \times \cdots \times M)^{hC_p} \ar{d} \ar{l} \\ B^{\rm cyc}(M^{\rm gp})^{hC_p} \ar{r} & (M^{\rm gp})^{hC_p} & M^{hC_p} \ar{l}  \end{tikzcd}\] 
is commutative. Forming horizontal pullbacks we obtain a canonical morphism $\psi_p \colon B^{\rm rep}(M) \to B^{\rm rep}(M)^{hC_p}$ compatible with the analogous map on cyclic bar constructions. From this we obtain the \emph{replete Frobenius} \[\varphi_p \colon {\Bbb S}[B^{\rm rep}(M)] \xrightarrow{{\Bbb S}[\psi_p]} {\Bbb S}[B^{\rm rep}(M)^{hC_p}] \xrightarrow{} {\Bbb S}[B^{\rm rep}(M)]^{hC_p} \xrightarrow{{\rm can}} {\Bbb S}[B^{\rm rep}(M)]^{tC_p}\] which is compatible with that on ${\rm THH}({\Bbb S}[M]) \cong {\Bbb S}[B^{\rm cyc}(M)]$ by construction. Note that this procedure canonically equips ${\Bbb S}[B^{\rm rep}(M)]$ with the structure of an $\mathbb{E}_\infty$-algebra in the category of  cyclotomic spectra.
\end{construction}
\begin{definition}
    We define the Frobenius on log topological Hochschild homology ${\rm THH}(A, M) = {\rm THH}(A) \otimes_{{\Bbb S}[B^{\rm cyc}(M)]} {\Bbb S}[B^{\rm rep}(M)]$ by base change along ${\Bbb S}[M] \to A$ of the replete Frobenius
    \[\varphi_p \colon {\Bbb S}[B^{\rm rep}(M)]  \to {\Bbb S}[B^{\rm rep}(M)]^{tC_p}\]
    of Construction \ref{constr:cycl_structure}. That is, it is defined by the composition
    \[\begin{tikzcd}
  {\rm THH}(A) \otimes_{{\Bbb S}[B^{\rm cyc}(M)]} {\Bbb S}[B^{\rm rep}(M)] \arrow[r] \arrow[rd, "\varphi_p"'] & {\rm THH}(A)^{tC_p} \otimes_{{\Bbb S}[B^{\rm cyc}(M)]^{tC_p}} {\Bbb S}[B^{\rm rep}(M)]^{tC_p} \ar{d} \\
     & {\rm THH}(A, M)^{tC_p},\end{tikzcd}
    \] 
   where the right-hand vertical map arises from lax monoidality of the Tate construction \cite[Theorem I.3.1]{NS18}.
\end{definition}
\begin{remark}
    The construction given above of the cyclotomic structure  on log topological Hochschild homology is  dependent upon the fact that we are working on (pre)-log ring spectra of the form $(A, M)$ for an $\mathbb{E}_\infty$-ring $A$ and a discrete monoid $M$ (so that we can directly borrow some computations from \cite{NS18}. See also \cite[p.888]{Ob18} for a sketch of the same construction).  The construction of the cyclotomic structure for more general log ring spectra is significantly more difficult, as discussed in future work of Rognes--Sagave--Schlichtkrull (see also \cite{Rog}). 
\end{remark}

\subsection{Descent properties} As a consequence of Proposition \ref{prop:loghhflatdescent} and Corollary \ref{cor:transitivitylogthh}, the argument of \cite[Corollary 3.4]{BMS19} applies to obtain:

\begin{proposition}\label{prop:logthhflatdescent} Let $(R, P)$ be a pre-log ring spectrum. Then the functors \begin{align*}{\rm THH}((-, -) / (R, P)), \quad {\rm TC}^-((-, -) / (R, P)), \\ {\rm THH}((-, -) / (R, P))_{hS^1}, \quad {\rm TP}((-, -) / (R, P))\end{align*} are log hlf sheaves. 
\end{proposition}

Similarly, Proposition \ref{prop:loghhrepletebasechange} applies to obtain: 

\begin{proposition}\label{prop:logthhrepletebasechange} Let $(R, P) \to (A, M)$ be a map of pre-log rings. Assume that there is a map $(A, M) \to (S, Q)$ with the property that the induced map $M^{\rm gp} \to Q^{\rm gp}$ is a surjection, and let $(A^{\rm rep}, M^{\rm rep}) \to (S, Q)$ denote the repletion of $(A, M) \to (S, Q)$. Then the canonical map \[A^{\rm rep} \otimes_A {\rm THH}((A, M) / (R, P)) \xrightarrow{} {\rm THH}((A^{\rm rep}, M^{\rm rep}) / (R, P)) \] is an equivalence. 
\end{proposition}
 \begin{definition} We define log ${\rm TC}^{-}((R, P) ; {\Bbb Z}_p), {\rm TP}((R, P) ; {\Bbb Z}_p)$ and ${\rm TC}((R, P) ; {\Bbb Z}_p)$ to be that of the cyclotomic spectrum ${\rm THH}((R, P) ; {\Bbb Z}_p)$: \[{\rm TC}^-((R, P) ; {\Bbb Z}_p) := {\rm THH}((R, P) ; {\Bbb Z}_p)^{hS^1}, \quad {\rm TP}((R, P) ; {\Bbb Z}_p) := {\rm THH}((R, P) ; {\Bbb Z}_p)^{tS^1}, \] \[{\rm TC}((R, P) ; {\Bbb Z}_p) := {\rm fib}({\rm TC^{-1}((R, P) ; {\Bbb Z}_p)} \xrightarrow{\varphi_p^{hS^1} - {\rm can}} {\rm TP}((R, P) ; {\Bbb Z}_p)).\] By definition, these recover the usual, non-logarithmic notions when $P$ is the trivial monoid. 
\end{definition}

\section{The log quasisyntomic site} We now pursue a logarithmic version of the quasisyntomic site introduced in \cite{BMS19}. A similar theory has been developed by Koshikawa--Yao \cite{KY22}. They pursue a log version of the log prismatic cohomology developed in \cite{BS22}; in particular, they do not discuss the Nygaard-complete version of log prismatic cohomology considered here.

\subsection{Recollection on quasisyntomic rings} Let us briefly recall the notion of quasisyntomic ring from \cite[Section 4]{BMS19} and the basic properties of the resulting site. As in \emph{loc.\ cit.} all rings will be implicitly  $p$-complete and with bounded $p^\infty$-torsion. By definition, a ring $A$ is \emph{quasisyntomic} if it has bounded $p^\infty$-torsion and ${\Bbb L}_{A / {\Bbb Z}_p}$ has $p$-complete ${\rm Tor}$-amplitude in $[-1, 0]$. 

\begin{definition}[{\cite[Definition 4.10]{BMS19}}] Let $f \colon A \to B$ be a map of commutative rings. We say that $f$ is \emph{quasisyntomic} if 

\begin{enumerate}
\item $A \to B$ is $p$-completely flat, i.e., $B\otimes_A A/p$ is concentrated in degree $0$ and  is a flat $A/p$-module,  and
\item ${\Bbb L}_{B/A}$ has $p$-complete ${\rm Tor}$-amplitude in $[-1, 0]$. 
\end{enumerate}

If, in addition, the underlying map $f \colon A \to B$ of commutative rings is $p$-completely faithfully flat (that is, $B\otimes_A A/p$ is also a faithfully flat $A/p$-module), we say that $f$ is a \emph{quasisyntomic cover}.

\end{definition}

\begin{definition}[{\cite[Definition 4.20]{BMS19}}] The ring $S$ is \emph{quasiregular semiperfectoid} if the following conditions are satisfied: 

\begin{enumerate}
\item $S$ is quasisyntomic;
\item $S$ admits a map from a perfectoid ring;
\item $S/pS$ is semiperfect; that is, its Frobenius is surjective.  
\end{enumerate}
We shall write ${\rm QRSPerfd}$ for the resulting category. 
\end{definition}

Recall from \cite[Lemma 4.27]{BMS19} that ${\rm QRSPerfd}^{\rm op}$ admits the structure of a site and that the collection of sheaves on ${\rm QSyn}^{\rm op}$ identify with the collection of sheaves on ${\rm QRSPerfd}^{\rm op}$ (with values in any presentable ${\infty}$-category ${\cal C}$): This is the content of \cite[Proposition 4.31]{BMS19}.  The resulting equivalence is denoted \[{\rm Shv}_{\cal C}({\rm QRSPerfd}^{\rm op}) \xrightarrow{\simeq} {\rm Shv}_{\cal C}({\rm QSyn}^{\rm op}), \quad F \mapsto F^\sqsupset\] and the sheaf $F^\sqsupset$ is called the \emph{unfolding} of $F$.

\subsection{Complete monoid rings} 
We shall adopt the following convention regarding monoid rings: For any commutative monoid $M$, we let ${\Bbb Z}_p \langle M \rangle$ denote the $p$-completion of the monoid ring ${\Bbb Z}_p[M]$. For a different $p$-complete base ring $R$,
we shall write $R\langle M \rangle$ for the $p$-complete tensor product $R \widehat{\otimes}_{{\Bbb Z}_p} {\Bbb Z}_p \langle M \rangle$.  

\subsection{(Semi)perfect monoids} Fix a prime number $p$. Given any commutative monoid $M$, we write \[F_M \colon M \to M, \quad m \mapsto m^p\] for its $p$-power map.

\begin{definition} We call a monoid 

\begin{enumerate}
\item \emph{perfect} if $F_M$ is an isomorphism, and
\item \emph{semiperfect} if $F_M$ is surjective.
\end{enumerate}
\end{definition} 

For example, the additive monoid ${\Bbb Q}_p/{\Bbb Z}_p$ is semiperfect, but not perfect. This definition gives rise to notions of direct and inverse limit perfections: \[M_{\rm perf} := {\rm colim}(M \xrightarrow{F_M} M \xrightarrow{F_M} \cdots), \quad M^\flat := {\rm lim}(\cdots \xrightarrow{F_M} M \xrightarrow{F_M} M).\] Observe that a monoid is semiperfect precisely when the canonical map from its \emph{tilt} $M^\flat \to M$ is surjective. 

The following elementary observation is essential to us:

\begin{lemma} Let $R$ be a perfectoid ring and let $P$ be a perfect monoid. 
Then the $p$-complete monoid ring $R \langle P \rangle$ is a perfectoid ring. 
\end{lemma}

\begin{proof} We make use of the equivalence between the category of perfect prisms and the category of perfectoid rings provided by \cite[Theorem 3.10]{BS22}. 
Let $(A_{\rm inf}(R^\flat), (d))$ be the perfect prism with $A_{\rm inf}(R^\flat)/(d) \cong R$. We claim that this gives rise to a  perfect prism $(A_{\rm inf}(R^\flat) \widehat{\otimes}_{{\Bbb Z}_p} {\Bbb Z}_p \langle P \rangle, (d))$. Indeed, the Frobenius on the $(p,d)$-complete $\delta$-ring $A_{\rm inf}(R^\flat) \widehat{\otimes}_{{\Bbb Z}_p} {\Bbb Z}_p \langle P \rangle$ remains bijective since $P$ is perfect, and its mod $d$ reduction is $R \widehat{\otimes}_{{\Bbb Z}_p} {\Bbb Z}_p \langle P \rangle =: R\langle P \rangle$, as desired. 
\end{proof}

\subsection{Quasisyntomic pre-log rings}\label{ssec:QSynlog} The following are generalizations of the notions of quasismooth and quasisyntomic ring maps to the logarithmic setting. Recall that we assume all rings to be $p$-complete. We say that a pre-log ring $(A, M)$ with $M$ integral is \emph{log quasismooth} if $A$ has bounded $p^\infty$-torsion and ${\Bbb L}_{(A, M) / {
\Bbb Z}_p}$ is $p$-completely flat, while $(A, M)$ is \emph{log quasisyntomic} if $A$ has bounded $p^\infty$-torsion and ${\Bbb L}_{(A, M) / {\Bbb Z}_p}$ has $p$-complete ${\rm Tor}$-amplitude in $[-1, 0]$. 
\begin{definition}[{Cfr. \cite[Definition 4.10(2, 3)]{BMS19}}] Let $(f, f^\flat) \colon (A, M) \to (B, N)$ be an integral map of integral pre-log rings whose underlying rings have bounded $p^\infty$-torsion. 
\begin{enumerate}
\item We say that $(f, f^\flat)$ is \emph{log quasismooth} if 
\begin{itemize}
\item[(i)] $A \to B$ is $p$-completely flat, and
\item[(ii)] ${\Bbb L}_{(B, N) / (A, M)}$ is $p$-completely flat.
\end{itemize}
If, in addition, the underlying map $f \colon A \to B$ of commutative rings is $p$-completely faithfully flat, we say that $(f, f^\flat)$ is a \emph{log quasismooth cover}.
\item We say that $(f, f^\flat)$ is \emph{log quasisyntomic} if 
\begin{itemize}
\item[(i)] $A \to B$ is $p$-completely flat, and
\item[(ii)]  ${\Bbb L}_{(B, N) / (A, M)}$ has $p$-complete ${\rm Tor}$-amplitude in $[-1, 0]$. 
\end{itemize}
If, in addition, the underlying map $f \colon A \to B$ of commutative rings is $p$-completely faithfully flat, we say that $(f, f^\flat)$ is a \emph{log quasisyntomic cover}.
\end{enumerate}
\end{definition}

\begin{remark}
    
The transitivity sequence of the Gabber cotangent complex implies that log quasismooth and log quasisyntomic are properties that are closed under composition. As the composition of faithfully maps are faithfully flat, 
the same goes for log quasismooth and quasisyntomic covers. 

Similarly, the flat base change property of the Gabber cotangent complex implies that these classes of maps are closed under base change; here we use that the completed derived base change along a $p$-completely flat map coincides with the ordinary completed base change as soon as all rings in question have bounded $p^\infty$-torsion, as observed in the proof of \cite[Lemma 4.16(2)]{BMS19}. 
\end{remark}
\subsection{Quasiregular semiperfectoid pre-log rings}\label{ssec:QRSPlog} We shall use the following notion of quasiregular semiperfectoid in the log setting: 

\begin{definition}\label{def:logqrsperfd} Let $(S, Q)$ be an integral pre-log ring. We say that $(S, Q)$ is \emph{log quasiregular semiperfectoid} if the following conditions are satisfied: 

\begin{enumerate}
\item the pre-log ring $(S, Q)$ is log quasisyntomic;
\item $S$ admits a map from a perfectoid ring;
\item $S/pS$ and $Q$ are semiperfect. 
\end{enumerate}
We shall write ${\rm lQRSPerfd}$ for the resulting category. 
\end{definition}
\begin{remark}
Throughout the rest of the paper, we shall freely use the following elementary observation to pass between monoid rings of the form $R[P]$ and their $p$-complete variants $R \langle P \rangle$:
If $R$ is $p$-complete of bounded $p^\infty$-torsion, then the monoid ring $R[P]$ is also of bounded $p^\infty$-torsion, as it is free over $R$. 
Consequently, the cotangent complex ${\Bbb L}_{R\langle P \rangle / R[P]}$ vanishes after derived $p$-completion, by derived reduction mod $p$, derived base change, and derived Nakayama. 
\end{remark}

\begin{remark}\label{rem:logperfectoid} We can always arrange for a log quasiregular semiperfectoid ring to receive a surjection from a ``pre-log perfectoid ring'' with underlying monoid the tilt $Q^\flat$, e.g., by replacing the surjection $R \to S$ by $R\langle Q^\flat \rangle \to S$. 
\end{remark}

Recall from \cite[Remark 4.21]{BMS19} that if $S$ is quasiregular semiperfectoid, then the cotangent complex $\L_{S/\Z_p}$ has $p$-complete Tor amplitude in degree $-1$. The same holds for their log counterparts: 

\begin{lemma}
\label{syntomic.13}
Let $(S, Q)$ be in ${\rm lQRSPerfd}$.
Then the Gabber cotangent complex $\L_{(S,Q)/\Z_p}$ has $p$-complete Tor amplitude in degree $-1$.
\end{lemma}
\begin{proof}
For every $x\in S$,
we have $x=y^p+pz$ for some $y,z\in S$ since $S/pS$ is semiperfect,
which implies that $dx=pdy^{p-1}+pdz$, that is, $\Omega^1_{(S/pS)/\mathbb{F}_p} =0$. 
Moreover, for every $q\in Q$
we have $q=pr$ for some $r\in Q$ since $Q$ is semiperfect,
which implies that $d\log q = pd\log r$.
Hence the multiplication map
\[
\cdot p
\colon
\Omega_{(S,Q)/\Z_p}^1
\to
\Omega_{(S,Q)/\Z_p}^1
\]
is surjective.
Since $\Omega_{(S,Q)/\Z_p}^1\cong \pi_0(\L_{(S,Q)/\Z_p})$,
we obtain by base change. 
\[
\pi_0(\L_{(S,Q)/\Z_p}\otimes_S S/pS)
\cong
0.
\]
Since $(S, Q)$ is log quasisyntomic by condition (1), we find that ${\Bbb L}_{(S, Q)/{\Bbb Z}_p}$ has $p$-complete ${\rm Tor}$-amplitude in degree $-1$, as desired. 
\end{proof}

\begin{lemma}
\label{syntomic.8}
Let $(S,Q)$ be in ${\rm lQRSPerfd}$,
and let $R \to S$ be a map with $R$ a perfectoid ring. 
then $\L_{(S,Q)/R}$ has $p$-complete Tor amplitude in degree $-1$.
\end{lemma}
\begin{proof}
From the transitivity sequence for $\Z_p\to R\to (S,Q)$, we obtain after (derived) base change to $S/pS$ a cofiber sequence
\[
\L_{R/\Z_p} \otimes_R S/pS
\xrightarrow{\alpha_{(S,Q)}}
\L_{(S,Q)/\Z_p} \otimes_S S/pS
\to
\L_{(S,Q)/R} \otimes_S S/pS
\]
The first two terms have $p$-complete Tor amplitude in degree $-1$ by \cite[Proposition 4.19(2)]{BMS19} and Lemma \ref{syntomic.13}. This immediately implies that $\pi_i(\L_{(S,Q)/R} \otimes_S S/pS\otimes_{S/pS} M) =0$ for any discrete $S/pS$-module $M$, $i \geq 3$ and $i \leq 0$. We are then left to show that the same holds for $\pi_2$, that is, it remains to show that $\beta_{(S,Q)}:=\pi_1(\alpha_{(S,Q)})$ is pure in the sense that it is injective after tensoring with any discrete $S/pS$-module (we borrow the terminology from \cite[Lemma 4.25]{BMS19}). The term $\L_{R/\Z_p}$ is the ordinary (non-log) cotangent complex, and after $p$-completion, 
it coincides with $\ker(\theta_R)/\ker(\theta_R)^2 = R[1]$ by \cite[Example 4.24]{BMS19}. Thus
we can describe $\beta_{(S,Q)}$ as
\[
\ker(\theta_R)/\ker(\theta_R)^2\otimes_R S/pS
\to
\pi_1(\L_{(S,Q)/\Z_p}\otimes_S S/pS),
\]
and the source is a finite free $S/pS$-module with formation compatible with base changes in $S$.
By Lemma \ref{syntomic.13},
the target is a flat $S/pS$-module.
By the criterion provided by \cite[Lemma 4.26]{BMS19}  it is enough to show that $\beta_{(S,Q)}\otimes_{S/pS} k$ is injective for all perfect fields $k$ with a map $S/pS\to k$ (note that in loc. cit.\ there is no perfectness assumption on $k$, but we can reduce to this case).
By  functoriality,
$\beta_{(S,Q)}\otimes_{S/pS}k$ factors $\beta_{(k,Q)}$.
Furthermore,
$\beta_{(k,Q)}$ factors $\beta_{(k,Q_\perf)}$.
Hence it remains to show that $\beta_{(k,Q_{\perf})}$ is injective.

In this case,
$k$ is perfectoid.
By \cite[Lemma 3.14]{BMS19},
we have $\L_{k/R}\simeq 0$.
There is a cocartesian square
\[
\begin{tikzcd}
\L_{\Z[Q_\perf]/\Z}\otimes_{\Z[Q_\perf]} k\ar[d]\ar[r]&
\L_{k/R}\ar[d]
\\
\L_{(\Z[Q_\perf],Q_\perf)/\Z}\otimes_{\Z[Q_\perf]} k\ar[r]&
\L_{(k,Q_\perf)/R},
\end{tikzcd}
\]
see \cite[Eq.\ 3.3]{BLPO}.
Together with Lemma \ref{lem:perfectvanishing} below, we deduce that $\L_{(k,Q_\perf)/R}$ vanishes.
This means that $\beta_{(k,Q_\perf)}$ is an isomorphism.
\end{proof}

\begin{lemma}  
\label{lem:perfectvanishing}
Let $P$ be a perfect monoid.
Then the modules \[\L_{\Z[P]/\Z}\otimes_{\Bbb Z} \Z/p\Z \quad \text{and} \quad \L_{(\Z[P],P)/\Z}\otimes_{\Bbb Z} \Z/p\Z\] vanish.
\end{lemma}

\begin{proof} This is essentially an observation taken from \cite{Bha12}. For the statement about the ordinary cotangent complex, the flat base change equivalence \[\L_{\Z[P]/\Z}\otimes_{\Bbb Z} \Z/p\Z = {\Bbb L}_{{\Bbb Z}[P]/{\Bbb Z}} \otimes_{{\Bbb Z}[P]} {\Bbb F}_p[P] \simeq {\Bbb L}_{{\Bbb F}_p[P]/{\Bbb F}_p}\] allows us to conclude by \cite[Corollary 3.8]{Bha12}. The proof for the statement about the Gabber cotangent complex is analogous, using instead \cite[Corollary 7.11]{Bha12}.
\end{proof}

In particular, we have:

\begin{corollary}\label{cor:perfectvanishing} Let $(R, P)$ be a pre-log ring with $R$ a perfectoid ring and $P$ a perfect monoid. Then the canonical map \[{\Bbb L}_{R / {\Bbb Z}_p} \xrightarrow{} {\Bbb L}_{(R, P) / {\Bbb Z}_p}\] is an equivalence after $p$-completion. 
\end{corollary}

\begin{proof} By the transitivity sequence for $\Z_p\to R\to (R, P)$, it suffices to prove that ${\Bbb L}_{(R, P) / R}$ vanishes after $p$-completion. This follows from Lemma \ref{lem:perfectvanishing} and the description of ${\Bbb L}_{(R, P) / R}$ as the pushout of the diagram \[{\Bbb L}_{R / R} \xleftarrow{} R \otimes_{{\Bbb Z}_p[P]} {\Bbb L}_{{\Bbb Z}_p[P] / {\Bbb Z}_p} \xrightarrow{} R \otimes_{{\Bbb Z}[P]} {\Bbb L}_{({\Bbb Z}_p[P], P) / {\Bbb Z}_p}\] of $R$-modules. 
\end{proof}

\begin{remark}\label{rem:perfectoidtrivial} In light of Corollary \ref{cor:perfectvanishing}, we shall typically consider perfectoid rings $R$ with trivial log structure. 
\end{remark}

As pointed out in Remark \ref{rem:logperfectoid}, if $(S, Q)$ is a log quasiregular semiperfectoid ring, there is a surjection to $S$ from a perfectoid ring of the form $R\langle Q^\flat \rangle$. Let $Q^{\flat, \rm rep} \to Q$ denote the repletion of the surjection $Q^\flat \to Q$. The following observation will be crucial in the sequel:

\begin{proposition}\label{prop:qrqsrep} With notation as in the previous paragraph, there is a canonical equivalence \[{\Bbb L}_{(S, Q) / R\langle Q^\flat \rangle} \simeq {\Bbb L}_{S / R\langle Q^{\flat, \rm rep} \rangle}\] of $S$-modules after $p$-completion. 
\end{proposition}

We stress that the right-hand side of the equivalence is an ordinary, non-logarithmic cotangent complex. The reader already familiar with the material of \cite{BMS19} may wonder why Proposition \ref{prop:qrqsrep} does not imply that our resulting theory of Nygaard-complete log prismatic cohomology coincides with the ordinary, non-logarithmic variant. The point is that the ring $R\langle Q^{\flat, \rm rep} \rangle$ need not be a perfectoid ring (and it is not, in general: this happens when the map $Q^\flat\to Q$ is already exact). See also Remark \ref{rmk:comment_Qflat_to_Q_exact} below. 

\begin{proof}[Proof of Proposition \ref{prop:qrqsrep}] There are equivalences \[ {\Bbb L}_{(S, Q) / R\langle Q^\flat \rangle} \xrightarrow{\simeq} {\Bbb L}_{(S, Q) / (R\langle Q^\flat \rangle, Q^{\flat})} \xrightarrow{\simeq} {\Bbb L}_{(S, Q) / (R\langle Q^{\flat, \rm rep}\rangle, Q^{\flat, \rm rep})} \xleftarrow{\simeq} {\Bbb L}_{S / R\langle Q^{\flat, \rm rep}\rangle}\] after derived $p$-completion by Corollary \ref{cor:perfectvanishing}, the $p$-complete variant of replete base change (Lemma \ref{lem:repletebasechange}), and the fact that the resulting map $Q^{\flat, \rm rep} \to Q$ is strict. . 
\end{proof}

 \begin{example}\label{exa:qrsperfd} With a small abuse of notation, write $\frac{1}{p}\N$ for the perfect submonoid of $\Q_{\geq 0}$ consisting of all elements of the form $\frac{a}{p^n}$ for $a,n\geq0$.  Let $R$ be a perfectoid ring. Then the $p$-complete monoid ring $R\langle x^{1/p^{\infty}}, y^{1/p^{\infty}} \rangle$ obtained from the monoid ring ${\Bbb Z}[\frac{1}{p}{\Bbb N}\oplus \frac{1}{p}{\Bbb N}] \cong {\Bbb Z}[x^{1/p^{\infty}}, y^{1/p^\infty}]$ is still a perfectoid ring, and comes equipped with a canonical pre-log structure \begin{equation}\label{logperfectoid}\frac{1}{p}{\Bbb N} \oplus \frac{1}{p} {\Bbb N} \to (R\langle x^{1/p^\infty}, y^{1/p^{\infty}} \rangle, \cdot).\end{equation} 
 Consider now the pushout $\frac{1}{p}{\Bbb N} \oplus_{\Bbb N} \frac{1}{p}{\Bbb N}$. It is a natural example of a semiperfect monoid, and this participates in a pre-log structure with its monoid ring \begin{equation}\label{logqrqs}\frac{1}{p}{\Bbb N} \oplus_{\Bbb N} \frac{1}{p}{\Bbb N} \to (R\langle \frac{1}{p}{\Bbb N} \oplus_{\Bbb N} \frac{1}{p}{\Bbb N} \rangle, \cdot) \cong (R\langle x^{1/p^\infty}, y^{1/p^{\infty}} \rangle/(x - y), \cdot)\end{equation} which we claim to be a log quasiregular semiperfectoid ring. Let us write $(S, Q)$ for the pre-log ring determined by \eqref{logqrqs}, and $\langle x^{1/p^{\infty}} \rangle, \langle y^{1/p^{\infty}} \rangle$ for the two copies of $\frac{1}{p}{\Bbb N}$.

As conditions (2) and (3) in Definition \ref{def:logqrsperfd} are clear, it only remains to show that $(S, Q)$ is log quasisyntomic, that is, we want to show that ${\Bbb L}_{(S, Q)/{\Bbb Z}_p}$ has $p$-complete ${\rm Tor}$-amplitude in $[-1, 0]$. 

Using the transitivity sequence \[S \otimes_{R\langle x^{1/p^{\infty}} \rangle} {\Bbb L}_{R\langle x^{1/p^{\infty}} \rangle / {\Bbb Z}_p} \xrightarrow{} {\Bbb L}_{(S, Q) / {\Bbb Z}_p} \xrightarrow{} {\Bbb L}_{(S, Q) / R \langle x^{1/p^{\infty}} \rangle}\] and the fact that ${\Bbb L}_{R \langle x^{1/p^{\infty}} \rangle / {\Bbb Z}_p}$ has $p$-complete ${\rm Tor}$-amplitude in degree $-1$ (with derived $p$-completion $R\langle x^{1/p^{\infty}} \rangle [1]$, since $R\langle x^{1/p^{\infty}} \rangle$ a perfectoid ring), we find that it suffices to prove that ${\Bbb L}_{(S, Q) / R\langle x^{1/p^{\infty}} \rangle}$ has $p$-complete ${\rm Tor}$-amplitude in $[-1, 0]$ in order to draw the same conclusion about ${\Bbb L}_{(S, Q) / {\Bbb Z}_p}$.

For this, we observe that there are equivalences \[{\Bbb L}_{(S, Q) / R\langle x^{1/p^{\infty}}\rangle } \xrightarrow{\simeq} {\Bbb L}_{(S, Q) / (R\langle x^{1/p^{\infty}}\rangle, \langle x^{1/p^{\infty}} \rangle)} \xleftarrow{\simeq} S \otimes_{R\langle y^{1/p^{\infty}}\rangle} {\Bbb L}_{(R\langle y^{1/p^{\infty}}\rangle, \langle y^{1/p^{\infty}} \rangle) / (R\langle t \rangle, \langle t \rangle)}\] after $p$-completion. The first equivalence is due to   Corollary \ref{cor:perfectvanishing} (and the transitivity sequence), while the second is an application of base change for the Gabber cotangent complex since $(S, Q)$ can be realized as the pushout of the diagram \[(R\langle x^{1/p^{\infty}}\rangle, \langle x^{1/p^{\infty}} \rangle) \xleftarrow{} (R\langle t \rangle, \langle t \rangle) \xrightarrow{} (R\langle y^{1/p^{\infty}} \rangle, \langle y^{1/p^{\infty}} \rangle).\]

It thus suffices to prove that ${\Bbb L}_{(R \langle y^{1/p^{\infty}} \rangle, \langle y^{1/p^{\infty}} \rangle) / (R\langle t \rangle, \langle t \rangle)}$ has derived $p$-completion equivalent to $R\langle y^{1/p^{\infty}} \rangle[1]$. This can be seen using the equivalence \[{\Bbb L}_{(R\langle y^{1/p^{\infty}}\rangle, \langle y^{1/p^{\infty}} \rangle) / (R\langle t \rangle, \langle t \rangle)} \xrightarrow{\simeq} (R\langle y^{1/p^{\infty}} \rangle \otimes_{R\langle t \rangle} {\Bbb L}_{(R\langle t \rangle, \langle t \rangle) / R})[1]\] arising from the transitivity sequence of the Gabber cotangent complex associated to the composite  $R \to (R\langle t \rangle, \langle t \rangle) \to (R\langle x^{1/p^{\infty}}\rangle, \langle x^{1/p^{\infty}} \rangle)$, since ${\Bbb L}_{(R\langle t \rangle, \langle t \rangle) / R} = R\langle t \rangle\{{\rm dlog}(t)\}$. 
\end{example}

\begin{example}\label{exa:postrepletion} Let us explain the repletion procedure in the context of Example \ref{exa:qrsperfd}. There is a surjection \[(R\langle x^{1/p^{\infty}}, y^{1 / p^{\infty}} \rangle, \frac{1}{p} {\Bbb N} \oplus \frac{1}{p}{\Bbb N}) \to (R\langle x^{1/p^{\infty}}, y^{1 / p^{\infty}} \rangle/(x - y), \frac{1}{p} {\Bbb N} \oplus_{\Bbb N} \frac{1}{p}{\Bbb N})\] with perfectoid source and log quasiregular semiperfectoid target. Passing to the replete base change in this case, we obtain the surjection \begin{equation}\label{postrepletion}R\langle x^{1/p^{\infty}}, y^{1 / p^{\infty}}, x/y, y/x \rangle \to R\langle x^{1/p^{\infty}}, y^{1 / p^{\infty}}\rangle/(x - y)\end{equation} whose source is not a perfectoid ring. That the repletion takes this form can either be seen by explicitly carrying out the relevant computation or by using a very slight variant of \cite[Remark 2.17(2)]{BLPO} and the standard chart for the associated blow-up. 
\end{example}
\begin{remark}
    Proposition \ref{prop:qrqsrep}, and the related results Proposition \ref{prop:loghhvsordinaryhh}, Theorem \ref{thm:logthhoneparameter}, and Proposition \ref{prop:filterlogthhcotcx} show that something very concrete is gained by working in the logarithmic setup: Despite the source of \eqref{postrepletion} not being perfectoid, we have every bit as much of control of the associated cotangent complex and (topological) Hochschild homology as in the setup of \cite{BMS19}.
\end{remark}

\subsection{The site ${\rm lQRSPerfd}^{\rm op}$} Analogously to \cite[Lemma 4.27]{BMS19}, we have:

\begin{lemma} The category ${\rm lQRSPerfd}^{\rm op}$ with log quasisyntomic covers forms a site.  
\end{lemma}

\begin{proof}
We only need to show that the class of log quasisyntomic covers is closed under pullbacks in $\lQRSPerfd^{\rm op}$.
For this,
let $(A,M)\to (B,N)$ be a log quasisyntomic cover in $\lQRSPerfd$, and let $(A,M)\to (A',M')$ be a map in $\lQRSPerfd$.
Then
the $p$-completed pushout $(A',M')\to (B',N'):=(A'\widehat{\otimes}_A B,M'\oplus_M N)$ is a log quasisyntomic cover, and we only need to show $(B',N')\in \lQRSPerfd$.

Since $(A',M')$ is log quasisyntomic,
then so is $(B',N')$. Moreover, $B'$ receives a map from a perfectoid ring (since $A'$ does).
Since $B'/pB'\cong B/pB\otimes_{A/pA} A'/pA'$,
$B'/pB'$ is semiperfect.
Furthermore,
$N'$ is semiperfect since it is a quotient of the semiperfect monoid $N\oplus M'$.
Hence $(B',N')\in \lQRSPerfd$, which concludes the proof. 
\end{proof}

\subsection{Existence of log quasisyntomic covers} We now aim to prove that every log quasisyntomic $(A, M)$ admits a cover by a log quasiregular semiperfectoid ring: 

\begin{proposition}\label{prop:logquasicover} Let $(A, M)$ be in ${\rm lQSyn}$. Then there exists a log quasisyntomic cover $(A, M) \to (S, Q)$, where $(S, Q)$ is log quasiregular semiperfectoid ring. 
\end{proposition}

\begin{proof}
There exists a free $p$-complete algebra $F:=\Z_p[\widehat{\{x_i\}_{i\in I\cup J}}]$ and surjections $F\to A$ and $\N^J\to M$.
As in the proof of \cite[Lemma 4.28]{BMS19}, we obtain the quasisyntomic cover $F\to F_\infty$ by formally adjoining all $p$-th roots of $p$ and $x_i$ for all $i\in I\cup J$ in the $p$-complete sense.
Observe that $F_\infty$ is a perfectoid ring.

We impose the pre-log structures $P:=\N^J \to F$ and $P_\infty:=\N_\perf^J \to F_\infty$ sending the $j$-th unit vector to $x_j$ for all $j\in J$.
Since $F$ is flat over $\Z_p$,
we have equivalences
\[
\L_{(F,P)/\Z_p}\otimes_{\Z} \Z/p\Z
\simeq
\L_{(\F_p[\N^{I\cup J}],\N^J)/\F_p}
\simeq
\bigoplus_{i\in I} {\Bbb F}_p\{dx_i\} \oplus\bigoplus_{j\in J} {\Bbb F}_p\{{\rm dlog}(x_j)\} 
\] 
In particular,
$\L_{(F,P)/\Z_p}$ is $p$-completely flat,
and hence $(F,P)\in \lQSyn$.
Consider now the cocartesian square defining $\L_{(F_\infty,P_\infty)/\Z_p}$.
\[
\begin{tikzcd}
\L_{\Z[P_\infty]/\Z}\otimes_{\Z[P_\infty]}F_\infty
\ar[d]\ar[r]&
\L_{F_\infty/\Z_p}\ar[d]
\\
\L_{(\Z[P_\infty],P_\infty)/\Z}\otimes_{\Z[P_\infty]}F_\infty
\ar[r]&
\L_{(F_\infty,P_\infty)/\Z_p}.
\end{tikzcd}
\]
We already know $F_\infty\in \QSyn$ (in fact, it is a perfectoid ring).
Together with Lemma \ref{lem:perfectvanishing}, we see that $\L_{(F_\infty,P_\infty)/\Z_p}$ has $p$-complete Tor amplitude in $[-1,0]$.
From the transitivity triangle for $\Z_p\to (F,P)\to (F_\infty,P_\infty)$, we deduce that $\L_{(F_\infty,P_\infty)/(F,P)}$ has $p$-complete Tor amplitude in $[-1,0]$.

The map $P\to P_\infty$ is integral by \cite[Proposition I.4.6.7]{Ogu18} since $\Z[P]\to \Z[P_\infty]$ is flat.
Furthermore, $F_\infty$ is $p$-completely faithfully flat over $F$.
It follows that $(F,P)\to (F_\infty,P_\infty)$ is a log quasisyntomic cover.
Let $(A,M)\to (S,Q)$ be the $p$-complete base change of $(F,P)\to (F_\infty,P_\infty)$ along $(F,P)\to (A,M)$,
where $Q:=P^\infty\oplus_P M$.
The integrality of $P\to P^\infty$ ensures that $Q$ is integral.
Since $(A,M)\in \lQSyn$,
we have $(S,Q)\in \lQSyn$.
Furthermore,
$S$ receives a map from the perfectoid ring $F_\infty$.
Since $F_\infty\to S$ and $P_\infty\to Q$ are surjective,
$S$ and $Q$ are semiperfect.
Hence we have $(S,Q)\in \lQRSPerfd$.
\end{proof}

\subsection{Unfolding in the log setting} We now aim to describe the process of \emph{unfolding} in the logarithmic context, in analogy with the equivalence of \cite[Proposition 4.31]{BMS19}. We will need the following statement, in analogy with \cite[Lemma 4.30]{BMS19}: 

\begin{lemma}
\label{syntomic.11}
Let $(A,M)\to (S,Q)$ be a log quasisyntomic cover in $\lQSyn$.
If $(S,Q)\in \lQRSPerfd$, then all terms in the \v{C}ech nerve $(S^\bullet,Q^\bullet)$ is in $\lQRSPerfd$.
\end{lemma}

\begin{proof}
Each term in the \v{C}ech nerve is log quasisyntomic and receives a map from a perfectoid ring since $S$ does. Now each term $Q^{\oplus_M i}$ of the \v{C}ech nerve of the underlying map of monoids is semiperfect, and the relevant statement for the underlying rings follows from the proof of \cite[Lemma 4.30]{BMS19}
\end{proof}

The proof of the following is now completely analogous to \cite[Proposition 4.31]{BMS19}: 

\begin{theorem}
\label{thm:unfolding}
For all presentable $\infty$-categories $\cV$, restriction along ${\rm lQRSPerd}^{\rm op} \to {\rm lQSyn}^{\rm op}$ induces an equivalence
\[{\rm Shv}(\lQSyn^{\rm op},\cV)
\simeq
{\rm Shv}(\lQRSPerfd^{\rm op},\cV)\] of $\infty$-categories. \qed
\end{theorem}

As in \cite{BMS19}, we shall write \[{\rm Shv}({\rm lQRSPerfd}^{\rm op}, \cV) \xrightarrow{} {\rm Shv}({\rm lQSyn}^{\rm op}, \cV), \quad F \mapsto F^{\sqsupset}\] for the resulting equivalence. We shall refer to the resulting sheaf $F^\sqsupset$ on ${\rm lQSyn}$ as the \emph{unfolding} of the sheaf $F$ on ${\rm lQRSPerfd}$. 

\begin{remark}[Variants of ${\rm lQSyn}$ and ${\rm lQRSPerfd}$] As in \cite{BMS19}, we will need to work with slight variants of the categories  ${\rm lQSyn}$ and ${\rm lQRSPerfd}$. For example, for any pre-log ring $(A, M)$, we write ${\rm lQSyn}_{(A, M)}$ for the category of $(A, M)$-algebras $(B, N)$ with $(B, N)$ log quasisyntomic, while we write ${\rm lQRSPerfd}_{(A, M)}$ for the full subcategory consisting of those $(B, N)$ that are log quasiregular semiperfectoid rings. Similarly, if $(A, M)$ is log quasisyntomic, we write ${\rm lqSyn}_{(A, M)}$ and ${\rm lqrsPerfd}_{(A, M)}$ for the categories of log quasisyntomic and log quasiregular semiperfectoid $(A, M)$-algebras, respectively. As explained in \cite[Section 4.5]{BMS19}, analogs of the results above continue to hold in these categories.
\end{remark}

\section{Log de Rham cohomology}\label{sec:logderham} We now aim to establish logarithmic analogs of the results of \cite[Section 5]{BMS19}. First, fix a log quasisyntomic pre-log ring $(R, P)$, and let $\widehat{{\rm DF}}(R)$ denote the complete filtered derived category of $R$, that is, the full subcategory of the filtered derived category ${\rm DF}(R) := {\rm Fun}(({\Bbb Z}, \ge), {\rm D}(R))$ of filtration complete objects.

\begin{lemma}\label{lem:hodgecompletederhamsheaf} The presheaf \[{\rm lQSyn}_{(R, P)} \to \widehat{{\rm DF}}(R), \quad (A, M) \mapsto (\widehat{L\Omega}_{(A, M) / (R, P)}, \widehat{L\Omega}^{\ge *}_{(A, M) / (R, P)})\] is a sheaf. 
\end{lemma}

\begin{proof} This is analogous to \cite[Example 5.11]{BMS19}: It boils down to checking that $(A, M) \mapsto (\wedge_A^i {\Bbb L}_{(A, M)/(R, P)})^\wedge_p$ is a sheaf, which follows from Theorem \ref{thm:gabberdescent}.  
\end{proof}

\begin{lemma}\label{lem:wedgepowersflat} Let $(S, Q) \in {\rm lqrsPerfd}_{(R, P)}$. The complex $\wedge^i_S {\Bbb L}_{(S, Q) / (R, P)}[-i]$ is $p$-completely flat for each $i \ge 0$.
\end{lemma}

\begin{proof}This is the analog of \cite[Lemma 5.14]{BMS19}. Let us first treat the case $i = 1$. By definition, the map $(R, P) \to (S, Q)$ is log quasisyntomic with $(S, Q)$ a log quasiregular semiperfectoid ring, so that ${\Bbb L}_{(S, Q) / (R, P)}$ has $p$-complete ${\rm Tor}$-amplitude in $[-1, 0]$. Combining the transitivity sequence for the composite $({\Bbb Z}_p, \{1\}) \to (R, P) \to (S, Q)$ with Lemma \ref{syntomic.13}, we deduce that ${\Bbb L}_{(S, Q)/(R, P)}$ has $p$-complete ${\rm Tor}$-amplitude concentrated in degree $-1$, as desired. 

To treat the general case, we use the equivalence \cite[Proposition 4.3.2.1]{Ill71} $\wedge^i_S {\Bbb L}_{(S, Q) / (R, P)}[-i] \simeq \Gamma_S^i ({\Bbb L}_{(S, Q)/(R, P)}[-1])$ and that derived divided powers preserve $p$-complete flatness (here, we denoted by $\Gamma^i_S(M)$ the $i$-th derived divided power algebra of an $S$-module $M$). 
\end{proof}

\begin{corollary}\label{cor:loghhoddvanish} Let $(S, Q) \in {\rm lqrsPerfd}_{(R, P)}$. Then the groups $\pi_*{\rm HH}((S, Q)/(R, P) ; {\Bbb Z}_p)$ vanish in odd degrees, while in even degrees $2i$ they identify with the $p$-completion of $\wedge^i {\Bbb L}_{(S, Q) / (R, P)}[-i]$.
\end{corollary}

\begin{proof} This follows from the log ${\rm HKR}$-filtration (Theorem \ref{thm:loghkrfilt}): The resulting spectral sequence \[E^2_{i, j} = \pi_i(\wedge^j_S ({{\Bbb L}}_{(S, Q) / (R, P)})^\wedge_p) \implies \pi_{i + j}{\rm HH}((S, Q) / (R, P) ; {\Bbb Z}_p)\] is concentrated in bidegrees $(i, i)$ by Lemma \ref{lem:wedgepowersflat}, which gives the result. 
\end{proof}

\begin{remark}\label{rem:perfectoidbase} Lemma \ref{lem:wedgepowersflat}, and consequently Corollary \ref{cor:loghhoddvanish}, hold in slightly greater generality. For example, 
\begin{enumerate}
\item if $(R, P) = ({\Bbb Z}_p, \{1\})$, we can take $(S, Q)$ to be any log quasiregular semiperfectoid ring. This is the content of Lemma \ref{syntomic.13}. 
\item If $R$ is a perfectoid ring with a trivial log structure, we can still take $(S, Q)$ to be any log quasiregular semiperfectoid ring over $R$. This is the content of Lemma \ref{syntomic.8}. 
\end{enumerate}
\end{remark}

\begin{corollary}\label{cor:hcminussheaf} The presheaf \[{\rm lqrsPerfd}_{(R, P)} \to D(R), \quad (A, M) \mapsto \pi_0{\rm HC}^{-}((A, M) / (R, P) ; {\Bbb Z}_p)\] is a sheaf. 
\end{corollary}

\begin{proof} The homotopy fixed point spectral sequence computing ${\rm HC}^{-}((A, M) / (R, P))$ collapses by Corollary \ref{cor:loghhoddvanish}, and we obtain a complete filtration of its $\pi_0$ with graded pieces $(\wedge_S^i {\Bbb L}_{(S, Q) / (A, M)})^\wedge_p[-i]$ so that we can argue as in the proof of Lemma \ref{lem:hodgecompletederhamsheaf}.  
\end{proof}

By Corollary \ref{cor:hcminussheaf}, we may consider the unfolding $\pi_0{\rm HC}^-((-, -) / (R, P) ; {\Bbb Z}_p)^{\sqsupset}$, a sheaf on ${\rm lqSyn}_{(R, P)}$. In analogy with \cite[Proposition 5.15]{BMS19}, we have:

\begin{proposition} There is a canonical identification \[\pi_0{\rm HC}^-((-, -) / (R, P) ; {\Bbb Z}_p)^{\sqsupset} \simeq \widehat{L\Omega}_{(-, -) / (R, P)}\] of sheaves on ${\rm lqSyn}_{(R, P)}$. 
\end{proposition}

\begin{proof} We now have all the ingredients to follow the steps of the proof of \cite[Proposition 5.15]{BMS19}. For the convenience of the reader, we spell this out here. 

Let $(S, Q) \in {\rm lqrsPerfd}_{(R, P)}$. As observed in the proof of Corollary \ref{cor:hcminussheaf}, the homotopy fixed points spectral sequence gives a filtration of $\pi_0{\rm HC}^-((S, Q) / (R, P) ; {\Bbb Z}_p)$ with graded pieces equivalent to $(\wedge_S^i {\Bbb L}_{(S, Q) / (A, M)})^\wedge_p[-i]$. Hence $\pi_0{\rm HC}^-((-, -) / (R, P) ; {\Bbb Z}_p)$ can be viewed as a sheaf on ${\rm lqrsPerfd}_{(R, P)}$ with values in $\widehat{{\rm DF}}(R)$, which we can unfold to a sheaf $F^{\sqsupset}$ on ${\rm lqSyn}_{(R, P)}$. 

Suppose $(A, M)$ is log quasismooth over $(R, P)$. Then the proof of Corollary \ref{cor:hcminussheaf} shows that $({\rm gr}^i F^{\sqsupset})(A, M) = (\Omega^i_{(A, M) / (R, P)}[-i])^{\wedge}_p$ . A theorem of Beilinson \cite[Theorem 5.4(3)]{BMS19} now implies that $F^{\sqsupset}(A, M)$ is given by a commutative differential graded $R$-algebra \[A \to (\Omega^1_{(A, M) / (R, P)})^{\wedge}_p \to (\Omega^2_{(A, M) / (R, P)})^{\wedge}_p  \to \cdots.\] By Proposition \ref{prop:logderham}, the differential is the log de Rham differential. 

This shows the statement for log quasismooth algebras, and hence for the respective left Kan extensions to all animated pre-log $(R, P)$-algebras. But we know that these left Kan extensions must agree with the original functors once restricted to ${\rm lqSyn}_{(R, P)}$, since this is true for their associated graded 
$({\Bbb L}_{(-, -) / (R, P)})^{\wedge}_p$, and therefore the result follows. 
\end{proof}

\begin{remark} We observe that this gives rise to a proof of a $p$-complete version of Theorem \ref{thm:logantieau}. See \cite[Section 5]{BMS19}. \end{remark}

\subsection{Further remarks on perfectoid bases} Let $(S, Q)$ be in ${\rm lQRSPerfd}$  and let $R$ be a perfectoid ring equipped with a surjection  $R \to S$. As pointed out in Remark \ref{rem:logperfectoid}, another natural choice of perfectoid ring surjecting onto $S$ is the monoid ring $R\langle Q^\flat \rangle$. By Theorem \ref{thm:loghkrfilt} and Corollary \ref{cor:perfectvanishing}, the canonical map \[{\rm HH}((S, Q) / R\langle Q^\flat \rangle ; {\Bbb Z}_p) \to {\rm HH}((S, Q) / (R\langle Q^\flat \rangle, Q^\flat) ; {\Bbb Z}_p)\] is an equivalence. By the $p$-complete variant of replete base change (Proposition \ref{prop:loghhrepletebasechange}) and transitivity for logarithmic Hochschild homology (Proposition \ref{lem:transitivitylogthh}), we obtain a further equivalence \[{\rm HH}((S, Q) / (R\langle Q^\flat \rangle, Q^\flat) ; {\Bbb Z}_p) \xrightarrow{\simeq} {\rm HH}((S, Q) / (R\langle Q^{\flat, \rm rep} \rangle, Q^{\flat, \rm rep}) ; {\Bbb Z}_p),\] where $Q^{\flat, \rm rep} \to Q$ is the repletion of the canonical map $Q^\flat \to Q$. As the resulting repletion map $Q^{\flat, \rm rep} \to Q$ is strict, this proves:

\begin{proposition}\label{prop:loghhvsordinaryhh} Let $(S, Q)$ be a log quasiregular semiperfectoid ring and let $R$ be a perfectoid ring surjecting to $S$. There is a canonical equivalence \[{\rm HH}((S, Q) / R\langle Q^\flat \rangle ; {\Bbb Z}_p) \simeq {\rm HH}(S / R\langle Q^{\flat, {\rm rep}}\rangle ; {\Bbb Z}_p)\] relating log Hochschild homology with ordinary Hochschild homology relative to the commutative ring $R\langle Q^{\flat, {\rm rep}} \rangle$. \qed
\end{proposition}

The attentive reader will notice that Proposition \ref{prop:loghhvsordinaryhh} can be directly deduced from Theorem \ref{thm:loghkrfilt} and Proposition \ref{prop:qrqsrep} (the paragraph after which we refer to further discussion on this type of result).

\section{Logarithmic ${\rm THH}$ over perfectoid bases} We now work towards the proof of Theorem \ref{thm:logthhdefhh} by establishing analogs of the results of \cite[Section 6]{BMS19} in the log setting. As one consequence we obtain Proposition \ref{prop:filterlogthhcotcx}, which gives a filtration of the log ${\rm THH}$ of object in ${\rm lQRSPerfd}$, 
whose graded pieces can be described in terms of the ordinary (i.e.,  non-logarithmic) cotangent complex.

\subsection{Log ${\rm THH}$ of log quasismooth rings}  We work with a perfectoid base $R$. Let $u$ be a generator of  $\ker{\theta_R}/\ker\theta_R^2=R[1]$, where $\theta_R$ is again Fontaine's map. As observed in Lemma \ref{syntomic.8}, this gives a trivialization of $(\L_{R/\Z_p})^{\wedge}_p$, and  the HKR Theorem gives then an isomorphism $\ker{\theta_R}/\ker\theta_R^2 \cong \pi_2\HH(R;\Z_p)$. 
By \cite[Theorem 6.1]{BMS19}, this extends to an isomorphism of graded rings $\pi_*{\rm THH}(R;\Z_p) \cong R[u]$.   

\begin{lemma}\label{lem:logthhtohhcof} Let $(A, M)$ be an $R$-algebra. There is an $S^1$-equivariant cofiber sequence \[{\rm THH}((A, M) ; {\Bbb Z}_p)[2] \xrightarrow{u} {\rm THH}((A, M) ; {\Bbb Z}_p) \xrightarrow{} {\rm HH}((A, M) / R ; {\Bbb Z}_p),
\] and hence corresponding cofiber sequences on homotopy fixed points and Tate constructions. 
\end{lemma}

\begin{proof} This follows by applying $- \otimes_{{\rm THH}(R ; {\Bbb Z}_p)} {\rm THH}((A, M) ; {\Bbb Z}_p)$ to the cofiber sequence \[{\rm THH}(R ; {\Bbb Z}_p)[2] \xrightarrow{u} {\rm THH}(R ; {\Bbb Z}_p) \xrightarrow{} {\rm HH}(R/R ; {\Bbb Z}_p)\] of \cite[Theorem 6.7]{BMS19} and Proposition \ref{lem:transitivitylogthh}. 
\end{proof}

By Lemma \ref{lem:logthhlinearization}, the canonical map\[{\rm THH}((A, M) ; {\Bbb Z}_p) \xrightarrow{} {\rm HH}((A, M) ; {\Bbb Z}_p)\] is an isomorphism in degrees less than or equal to two. In particular, we may mimic  \cite[Construction 6.8]{BMS19} to obtain a map \begin{equation}\label{derhamtothh}(\Omega^*_{(A, M) / {\Bbb Z}})^\wedge_p \xrightarrow{} \pi_*{\rm THH}((A, M) ; {\Bbb Z}_p)\end{equation}
of graded $A$-algebras. The following is then a version of \cite[Corollary 6.9]{BMS19}, 
with an analogous proof:

\begin{corollary}\label{cor:derhamtothh} The map \eqref{derhamtothh} linearizes to a map \[(\Omega^*_{(A, M) / {\Bbb Z}})^\wedge_p \otimes_R \pi_*{\rm THH}(R ; {\Bbb Z}_p) \xrightarrow{} \pi_*{\rm THH}((A, M) ; {\Bbb Z}_p)\] which is an isomorphism if $(A, M)$ is log quasismooth. \qed
\end{corollary}

From this, we formally obtain an analog of \cite[Corollary 6.10]{BMS19}: The functor ${\rm THH}((-, -) ; {\Bbb Z}_p)$ on pre-log $R$-algebras whose underlying ring is $p$-complete admits a filtration with graded pieces shifts of wedge powers of the log cotangent complex $({\Bbb L}_{(-, -) / R})^\wedge_p$; the sum of the degree of the wedge power and the degree shift is even. 

This, in particular, gives rise to a proof of Theorem \ref{thm:logthhdefhh}, 
the setting and statement we recall here.

Let $(S, Q)$ be a log quasiregular semiperfectoid pre-log ring and let $R$ be a perfectoid ring surjecting to $S$. Consider the surjection $R\langle Q^\flat \rangle \to S$, where $Q^\flat$ is the tilt of $Q$.

\begin{theorem}\label{thm:logthhoneparameter} With notation as in the previous paragraph, there is a cofiber sequence \[{\rm THH}((S, Q) ; {\Bbb Z}_p)[2] \xrightarrow{u} {\rm THH}((S, Q) ; {\Bbb Z}_p) \xrightarrow{} {\rm HH}(S / R\langle Q^{\flat, \rm rep} \rangle ; {\Bbb Z}_p).\]
\end{theorem}

\begin{proof} 
Applying Lemma \ref{lem:logthhtohhcof} to the case at hand, we obtain a cofiber sequence \[{\rm THH}((S, Q) ; {\Bbb Z}_p)[2] \xrightarrow{u} {\rm THH}((S, Q) ; {\Bbb Z}_p) \xrightarrow{} {\rm HH}((S, Q) / R\langle Q^\flat \rangle ; {\Bbb Z}_p).\] Now we apply Proposition \ref{prop:loghhvsordinaryhh} to the cofiber term ${\rm HH}((S, Q) / R \langle Q^\flat \rangle ; {\Bbb Z}_p)$ to identify it with the ordinary Hochschild homology ${\rm HH}(S / R\langle Q^{\flat, \rm rep} \rangle ; {\Bbb Z}_p)$, which gives the desired result. 
\end{proof}

\begin{remark}\label{rem:weakpostnikov} To make the slogan that log topological Hochschild homology is a one-parameter deformation of ordinary Hochschild homology slightly more precise, observe that \[\{\tau_{\le n} {\rm THH}(R\langle Q^\flat \rangle ; {\Bbb Z}_p) \otimes_{{\rm THH}(R\langle Q^\flat \rangle ; {\Bbb Z}_p)} {\rm THH}((S, Q) ; {\Bbb Z}_p)\}\] is a \emph{weak Postnikov tower} in the sense of \cite[Section 3]{BMS19}, which can be proved using the techniques therein. If lifting problems along this tower could be phrased in terms of (logarithmic) deformation theory, one could attempt to use this to establish a universal property for logarithmic topological Hochschild homology, as alluded to in the introduction.  
\end{remark}

\begin{remark}\label{rmk:comment_Qflat_to_Q_exact}
    Note that if $Q^\flat \to Q$ was already exact, then $R\langle Q^{\flat, \rm rep} \rangle =  R\langle Q^{\flat} \rangle$ would be a perfectoid ring, and by \cite[Theorem 6.7]{BMS19}, ${\rm HH}(S / R\langle Q^{\flat, \rm rep} \rangle ; {\Bbb Z}_p)$ would deform to the non-log version ${\rm THH}(S ; {\Bbb Z}_p)$. In particular, the canonical map ${\rm THH}(S ; {\Bbb Z}_p) \to {\rm THH}((S, Q) ; {\Bbb Z}_p)$ is an equivalence in this case.
\end{remark}
\begin{proposition}\label{prop:filterlogthhcotcx} With the notation of the previous paragraph, the log topological Hochschild homology ${\rm THH}((S, Q) ; {\Bbb Z}_p)$ admits a complete descending filtration with graded pieces \[ \bigoplus_{\substack{0\leq i\leq n\\ i-n\ \mathrm{even}}} (\wedge^i_S {\Bbb L}_{S / R\langle Q^{\flat, \rm rep} \rangle})^\wedge_p[n],\] where $Q^{\flat, \rm rep} \to Q$ denotes the repletion of the canonical map $Q^\flat \to Q$. 
\end{proposition}

\begin{proof}  Since $R\langle Q^\flat \rangle$ is a perfectoid ring, the analog of \cite[Corollary 6.10]{BMS19} described above applies to give a filtration with graded pieces  \[ \bigoplus_{\substack{0\leq i\leq n\\ i-n\ \mathrm{even}}} (\wedge^i_S {\Bbb L}_{(S, Q) / R\langle Q^\flat \rangle})^\wedge_p[n].\] Now the result follows from Proposition \ref{prop:qrqsrep}.  
\end{proof}

\begin{remark}
    Note that the isomorphism $\pi_* \THH(R;\Z_p)$ for $R$ a perfectoid ring  can be interpreted as a generalization of B\"okstedt fundamental computation $\pi_* \THH(\F_p ; \Z_p) \cong \F_p[u]$, where $u$ is a generator of $\pi_2\THH(\F_p;\Z_p)$. In fact, the proof in \cite{BMS19} relies on B\"okstedt's computation, in the sense that it is ultimately reduced to the case $R=\F_p$ for $R$ in characteristic $p>0$ (followed by a long devissage in mixed characteristic). Calling the generator $u$ a Bott element would be tempting, 
    but see \cite[Section 1.2]{Hess_Nik}.  
\end{remark}

\begin{remark}[On Breuil--Kisin twists]\label{rem:breuilkisin} As explained in Remark \ref{rem:perfectoidtrivial}, we shall typically consider perfectoid rings $R$ with trivial pre-log structure. In particular, there is a natural isomorphism \[\pi_*{\rm TP}(R; {\Bbb Z}_p) = \bigoplus_{i \in I} A_{\rm inf}\{i\},\] see \cite[Section 6.2]{BMS19}, and we shall follow the same conventions as \cite{BMS19} with regards to (omitting) Breuil--Kisin twists. 
\end{remark}

\section{Filtrations and log prismatic cohomology} We now aim to prove Theorem \ref{thm:mainthm} from the introduction.

\subsection{Evenness of log ${\rm THH}$ of quasiregular semiperfectoids} Let $R$ be a perfectoid ring. As a consequence of our work so far, we obtain analogs of \cite[Theorems 7.1 and 7.2]{BMS19} with similar proofs. For the convenience of the reader, we choose to spell this out here: 

\begin{theorem}\label{thm:bms19thm71} Let $(S, Q)$ be  in ${\rm lQRSPerfd}_R$. 
\begin{enumerate}
\item The homotopy groups $\pi_*{\rm THH}((S, Q) ; {\Bbb Z}_p)$ are concentrated in even degrees. 
\item Multiplication by $u \in \pi_2{\rm THH}(R ; {\Bbb Z}_p)$ induces an injection \[\pi_{2i - 2}{\rm THH}((S, Q) ; {\Bbb Z}_p) \xrightarrow{u} \pi_{2i}{\rm THH}((S, Q) ; {\Bbb Z}_p).\]
\item Consider $\pi_0{\rm THH}((S, Q) ; {\Bbb Z}_p)[u^{\pm 1}]$ as a  filtered commutative $R$-algebra, equipped with an increasing filtration. Then there is a canonical isomorphism \[(\Gamma_S^* (\pi_1{\Bbb L}_{(S, Q) / R})^\wedge_p)^\wedge_p \cong {\rm gr}_*(\pi_0 {\rm THH}((S, Q) ; {\Bbb Z}_p)[u^{\pm 1}])\] of graded rings. On the left-hand side, the first $p$-completion is of an $S$-module, the second of a graded ring. 
\item Each even homotopy group $\pi_{2i}{\rm THH}((S, Q) ; {\Bbb Z}_p)$ is $p$-completely flat over $S$. 
\end{enumerate}
\end{theorem}

\begin{proof} We first prove (1). As remarked after the proof of Corollary \ref{cor:derhamtothh}, the log topological Hochschild homology ${\rm THH}((S, Q) ; {\Bbb Z}_p)$ admits a complete filtration with graded pieces a sum of shifts of wedge powers of the log cotangent complex $({\Bbb L}_{(S, Q) / R})^{\wedge}_p$; moreover, the sum of the degree of the shift and the wedge power is even. By Lemma \ref{lem:wedgepowersflat}, each wedge power $(\wedge_S^i {\Bbb L}_{(S, Q) / R})^{\wedge}_p$ has $p$-complete ${\rm Tor}$-amplitude in (homological) degree $i$, and hence it lives in degree $i$ by \cite[Lemma 4.7]{BMS19}. Hence the associated graded of the complete filtration of ${\rm THH}((S, Q) ; {\Bbb Z}_p)$ lives in even degrees, which proves (1).

We now prove (2) and (3). Consider the cofiber sequence \[{\rm THH}((S, Q) ; {\Bbb Z}_p)[2] \xrightarrow{u} {\rm THH}((S, Q) ; {\Bbb Z}_p) \xrightarrow{} {\rm HH}((S, Q) / R) ; {\Bbb Z}_p)\] of Lemma \ref{lem:logthhtohhcof}. The associated long exact sequence splits into short exact sequences \[0 \xrightarrow{} \pi_{2i - 2} {\rm THH}((S, Q) ; {\Bbb Z}_p) \xrightarrow{u} \pi_{2i} {\rm THH}((S, Q) ; {\Bbb Z}_p) \xrightarrow{} \pi_{2i} {\rm HH}((S, Q) / R ; {\Bbb Z}_p) \to 0\] since both $\pi_*{\rm THH}((S, Q) ; {\Bbb Z}_p)$ and $\pi_*{\rm HH}((S, Q) / R ; {\Bbb Z}_p)$ live in even degrees. This proves (2), and part (3) follows from the isomorphism $\pi_{2i} {\rm HH}((S, Q) / R ; {\Bbb Z}_p) \cong \pi_i(\wedge_S^i {\Bbb L}_{(S, Q) / R})^\wedge_p$ from Corollary \ref{cor:loghhoddvanish}, as well as the isomorphism $\pi_i(\wedge_S^i {\Bbb L}_{(S, Q) / R})^\wedge_p \cong (\Gamma_S^i \pi_1 ({\Bbb L}_{(S, Q) / R})^\wedge_p)^\wedge_p$ (note that the right-hand side is discrete since the cotangent complex is a $p$-completely flat module after $p$-completion). 

Since $ (\Gamma_S^i \pi_1 ({\Bbb L}_{(S, Q) / R})^\wedge_p)^\wedge_p$ is $p$-completely flat (see \cite[Proof of Theorem 7.1]{BMS19}), this also proves part (4). 
\end{proof}

\subsection{Definition of Nygaard-complete log prismatic cohomology}\label{subsec:nygaardcomplete} As a consequence of Theorem \ref{thm:bms19thm71}, for a log quasiregular semiperfectoid $R$-algebra $(S, Q)$, the homotopy fixed points and Tate spectral seqences computing $\pi_*{\rm TC}^-((S, Q) ; {\Bbb Z}_p)$ and $\pi_*{\rm TP}((S, Q) ; {\Bbb Z}_p)$ degenerate. This equips the \emph{Nygaard-complete log prismatic cohomology} \[\widehat{\Prism}_{(S, Q)} := \pi_0{\rm TC}^-((S, Q) ; {\Bbb Z}_p) \cong \pi_0{\rm TP}((S, Q) ; {\Bbb Z}_p)\] with a complete, descending filtration ${\rm Fil}_N^{\ge \bullet} \widehat{\Prism}_{(S, Q)}$, that we call the \emph{logarithmic Nygaard filtration}. The associated graded ${\rm gr}_N^i \widehat{\Prism}_{(S, Q)}$ identifies with log topological Hochschild homology 
\[{\rm gr}_N^i \widehat{\Prism}_{(S, Q)}\cong \pi_{2i} {\rm THH}((S, Q) ; {\Bbb Z}_p),\] and the $i$th filtration level ${\rm Fil}^{\ge 2i}_N \widehat{\Prism}_{(S, Q)}$ identifies with $\pi_{2i} {\rm TC}^-((S, Q) ; {\Bbb Z}_p)$, where the inclusion to $\widehat{\Prism}_{(S, Q)} := \pi_0{\rm TC}^-((S, Q) ; {\Bbb Z}_p)$ identifies with multiplication by $v^i \in \pi_{-2i} {\rm TC}^-(R; {\Bbb Z}_p)$. As in \cite[Theorem 7.2(4), (5)]{BMS19},  the cyclotomic Frobenius from log ${\rm TC}^{-}$ to log ${\rm TP}$ gives rise to divided Frobenii \[\varphi_{S, i} \colon {\rm Fil}^{\ge i}_N \widehat{\Prism}_{(S, Q)} \to \widehat{\Prism}_{(S, Q)},\] and there is a natural isomorphism $\widehat{\Prism}_{(S, Q)}/\xi \cong \widehat{L\Omega}_{(S, Q) / R}$.

The above paragraph summarizes the log version of \cite[Theorem 7.2]{BMS19}. 

\begin{remark}[Relationship with Koshikawa--Yao]\label{rem:koshikawayao} Let $S$ be a quasiregular semiperfectoid ring. Recall that Bhatt--Scholze \cite{BS22} defines prismatic cohomology $\Prism_S$ via the prismatic site. 
They prove \cite[Theorem 13.1]{BS22} that a canonical comparison map $\Prism_S \to \widehat{\Prism}_S$ exhibits the target as the Nygaard completion of the source, where the Nygaard filtration is defined directly on $\Prism_S$. 

Koshikawa \cite{Kos22} introduced the log prismatic site. In work with Yao \cite{KY}, they define a certain Nygaard filtration on (a derived variant of) log prismatic cohomology $\Prism_{(S, Q)}$ for $(S, Q)$ a log quasiregular semiperfectoid pre-log ring. We now review one key idea for their setup and explain how it translates into the present paper's language.  

Recall that, if $S$ is a quasiregular semiperfectoid ring, then the prismatic cohomology $\Prism_{S/A_{\rm inf}(R)} := \Prism_{S/(A_{\rm inf}(R), {\rm ker}(\theta))}$ is initial \cite[Proposition 7.10]{BS22}, which gives rise to a comparison map $\Prism_S \to \widehat{\Prism}_S$. 

In the logarithmic setting, we may consider the object $\Prism_{(S, Q) / A_{{\rm inf}}(R)}$, and try to define and study a Nygaard filtration on it in hopes of obtaining an analog of Bhatt--Scholze's comparison result.

For this, it seems that Koshikawa--Yao are using techniques similar to our Proposition \ref{prop:qrqsrep}, Proposition \ref{prop:loghhvsordinaryhh}, and Theorem \ref{thm:logthhoneparameter} to describe their log prismatic cohomology as ordinary prismatic cohomology, at the price of working with a non-perfect base prism. For this reason, we believe that our Theorem \ref{thm:logthhoneparameter} and the surrounding techniques will be the key ingredient in comparing our Nygaard-complete log prismatic cohomology with that considered by Koshikawa--Yao, and we hope to elaborate upon this in the future. 
\end{remark}

\subsection{Unfolding $\pi_{2i}{\rm THH}((-, -) ; {\Bbb Z}_p)$} The following is  analogous to \cite[Construction 7.4, Proposition 7.5]{BMS19}. Let $R$ be a perfectoid ring and $(S, Q)$ be a log quasiregular semiperfectoid $R$-algebra. By Theorem \ref{thm:bms19thm71}, the $S$-module $\pi_{2i}{\rm THH}((S, Q) ; {\Bbb Z}_p)$ admits a finite filtration with graded pieces shifts of wedge powers of the log cotangent complex $({\Bbb L}_{(S, Q) / R})^{\wedge}_p$. In particular, Theorem \ref{thm:gabberdescent} implies that the presheaf \[{\rm lQRSPerfd}_R \to D(R), \quad (S, Q) \mapsto \pi_{2i}{\rm THH}((S, Q) ; {\Bbb Z}_p)\] is a sheaf. By unfolding, we obtain a sheaf \[{\rm lQSyn}_R \to D(R), \quad (S, Q) \mapsto \pi_{2i}{\rm THH}((S, Q) ; {\Bbb Z}_p)^\sqsupset,\] and so we get the following analog of \cite[Proposition 7.5]{BMS19}:

\begin{proposition}\label{prop:unfoldinglogthh} Let $(A, M)$ be a log quasisyntomic $R$-algebra. Then the log topological Hochschild homology ${\rm THH}((A, M) ; {\Bbb Z}_p)$ admits a complete, descending, $S^1$-equivariant filtration such that each graded piece ${\rm gr}^i {\rm THH}((A, M) ; {\Bbb Z}_p)$ (with trivial $S^1$-action) itself admits a filtration with graded pieces $(\wedge_A^j {\Bbb L}_{(A, M) / R})^\wedge_p[2i - j]$, where $0 \le j \le i$. 
\end{proposition}

\begin{proof} If $(A, M)$ is a log quasiregular semiperfectoid pre-log ring, the desired filtration is the double speed Postnikov filtration. In general, we unfold and apply Theorem \ref{thm:gabberdescent} and Proposition \ref{prop:logthhflatdescent}.  
\end{proof}

\subsection{Unfolding $\pi_0{\rm TC}^-((-, -) ; {\Bbb Z}_p)$}  By the discussion after the proof of Theorem \ref{thm:bms19thm71}, there is a sheaf \[{\rm lQRSPerd}_R \to \widehat{{\rm DF}}(A_{\rm inf}(R)), \quad (S, Q) \mapsto (\widehat{\Prism}_{(S, Q)}, {\rm Fil}_N^{\ge \bullet} \widehat{\Prism}_{(S, Q)})\] which unfolds to a sheaf $(\widehat{\Prism}_{(-, -)}, {\rm Fil}_N^{\ge \bullet} \widehat{\Prism}_{(-, -)})$ defined on all log quasisyntomic $R$-algebras. As in \cite[Proposition 7.8]{BMS19}, we have:

\begin{proposition} Let $(A, M)$ be a log quasisyntomic $R$-algebra. Each graded piece ${\rm gr}_N^i \widehat{\Prism}_{(A, M)}$ admits a finite filtration with graded pieces given by $(\wedge_A^j {\Bbb L}_{(A, M) / R})^\wedge_p [-j]$ for $0 \le j \le i$. 
\end{proposition}

\begin{proof} By construction (see the discussion after the proof of Theorem \ref{thm:bms19thm71}), the graded piece ${\rm gr}_N^i \widehat{\Prism}_{(-, -)}$ is the unfolding $\pi_{2i}{\rm THH}((-, -) ; {\Bbb Z}_p)^{\sqsupset}$. The result follows from (the discussion prior to) Proposition \ref{prop:unfoldinglogthh}. 
\end{proof}

\subsection{Setting up the motivic filtrations} 
We are finally able to prove Theorem \ref{thm:mainthm} from the introduction.
As in the exposition of \cite[Section 7.3]{BMS19}, 
we begin by working with a fixed perfectoid base ring $R$: 

\begin{proposition}\label{prop:mainthmperfectoidbase} Let $(A, M)$ be a log quasisyntomic $R$-algebra.

\begin{enumerate}
\item The negative log topological cyclic homology ${\rm TC}^-((A, M) ; {\Bbb Z}_p)$ admits a complete and exhaustive filtration with $i$th graded piece ${\rm Fil}_N^{\ge i}\widehat{\Prism}_{(A, M)}[2i]$. 
\item The log topological periodic homology ${\rm TP}((A, M) ; {\Bbb Z}_p)$ admits a complete and exhaustive filtration with $i$th graded piece $\widehat{\Prism}_{(A, M)}[2i]$. 
\end{enumerate} 
\end{proposition}

\begin{proof} We follow the proof strategy of \cite[Proposition 7.13]{BMS19}. Consider first the sheaf $(S, Q) \mapsto \tau_{\ge 2n}{\rm TC}^-((S, Q) ; {\Bbb Z}_p)$ for log quasiregular semiperfectoid $(S, Q)$: This is indeed a sheaf by Theorem \ref{thm:bms19thm71} and the ensuing discussion. Hence it unfolds to a sheaf $\tau_{\ge 2n}{\rm TC}^-((-, -) ; {\Bbb Z}_p)^{\sqsupset}$ on all log quasisyntomic $R$-algebras. Varying $n$, we obtain a sheaf \[\tau_{\ge 2*}{\rm TC}^-((-, -) ; {\Bbb Z}_p)^{\sqsupset} \colon {\rm lQSyn}_R \to \widehat{{\rm DF}}(A_{\rm inf}(R)).\] The $i$th graded piece is $(\pi_{2i}{\rm TC}^-((-, -) ; {\Bbb Z}_p)[2i])^{\sqsupset}$, which is canonically isomorphic to ${\rm Fil}_N^{\ge i} \widehat{\Prism}_{(-, -)}[2i]$. This finishes the proof of (1), and the proof of (2) is analogous. 
\end{proof}

\begin{proof}[Proof of Theorem \ref{thm:mainthm}] We follow the proof strategy of \cite[Theorem 1.12]{BMS19}. Parts (1) and (2) follow from Proposition \ref{prop:logquasicover}, Theorem \ref{thm:bms19thm71}, and the discussion following its proof. For part (3), we reduce to the case of a perfectoid base ring by Proposition \ref{prop:logquasicover}, in which case this follows from Proposition \ref{prop:mainthmperfectoidbase}. From this, the last part follows formally. \end{proof}

\section{Sample application: the de Rham comparison}

In future work, we intend to prove a log analog of \cite[Theorem 13.1]{BS22} relating our construction with that of \cite{KY22}. From this, we can transport the various comparison results relating log prismatic cohomology to log de Rham, log crystalline, and log \'etale cohomology from those of Koshikawa--Yao. This will provide an extension of the results in \cite{BS22} to the bad reduction (semistable) case.

As a reality-check, we demonstrate how to obtain the log de Rham comparison (up to a  Frobenius twist, see Remark \ref{rem:logsegal}) directly from our setup, as in \cite[Section 11]{BMS19}. Let us keep the notation of \cite[Section 11]{BMS19}. We fix a discretely valued extension $K$ of ${\Bbb Q}_p$. We let ${\cal O}_K$ be its ring of integers, $k$ its residue field, and we fix $\bar{\omega}$  a uniformizer of ${\cal O}_K$. We write $K_{\infty}$ for the $p$-adic completion of $K(\bar{\omega}^{1/p^{\infty}})$. There is a canonical map $W(k)[[z]] \to {\cal O}_K$ sending $z$ to the uniformizer $\bar{\omega}$, whose kernel is generated by an Eisenstein polynomial $E(z)$.  

Recall that an integral map of integral monoids is of \emph{Cartier type} if its relative Frobenius is exact. If $(A, M)$ is a pre-log ring with $A$ an adic ring, classically complete with respect to a finitely generated ideal $I$, we shall write (following \cite[Appendix A]{Kos22}) $({\rm Spf}(A), M)^a$ for the corresponding log formal scheme.

\begin{theorem}\label{thm:derhamcomparison} 
Let $(A, M)$ be an integral pre-log ring and assume that $({\rm Spf}(A), M)^a$ is log smooth of Cartier type over $({\cal O}_K, \langle \pi \rangle)$. Scalar extension along $W(k)[[z]] \to {\cal O}_K$ induces an isomorphism
\[\widehat{\Prism}_{(A, M) / ((W(k)[[z]], \langle z \rangle)} \otimes_{W(k)[[z]]} {\cal O}_K \simeq (\Omega_{(A, M) / ({\cal O}_K, \langle \pi \rangle)})^\wedge_p\] of ${\Bbb E}_{\infty}$-${\cal O}_K$-algebras. 
\end{theorem}

We refer to \cite[Appendix A]{Kos22} for material on log formal schemes. 

\begin{remark}\label{rem:logsegal} Theorem \ref{thm:derhamcomparison} is an analog of \cite[Corollary 11.12(ii)]{BMS19}. The descent needed for the missing Frobenius twist relies on a log version of the Segal conjecture (analogously to how \cite[Proposition 11.15]{BMS19} depends upon \cite[Corollary 8.18]{BMS19}). This will be established in forthcoming joint work with Alberto Merici.
\end{remark}

\subsection{Log cyclotomic bases} As in \cite[Section 11]{BMS19}, we will need to work with the relative version of logarithmic ${\rm THH}$. As recorded in Lemma \ref{lem:transitivitylogthh}, we have the transitivity formula \[{\rm THH}((A, M) / (R, P)) \simeq R \otimes_{{\rm THH}(R, P)} {\rm THH}(A, M).\] For this relative ${\rm THH}$ to inherit a cyclotomic structure from its absolute building blocks, we therefore need the following compatibility condition (cf.\ \cite[Definition 3.2.1]{HRW22}): 

\begin{definition} A \emph{log cyclotomic base} is a pre-log ring spectrum $(R, P)$ together with a commutative diagram \[\begin{tikzcd}{\rm THH}((R, P) ; {\Bbb Z}_p) \ar{r}{\varphi_p} \ar{d} & {\rm THH}((R, P) ; {\Bbb Z}_p)^{tC_p} \ar{d} \\ R \ar{r} & R^{tC_p}\end{tikzcd}\] of $S^1$-equivariant ${\Bbb E}_{\infty}$-rings. 
\end{definition}

\begin{lemma} The pre-log ring spectrum $({\Bbb S}[z], \langle z \rangle)$ is a log cyclotomic base. 
\end{lemma}

\begin{proof} This is similar to \cite[Proposition 11.3]{BMS19}, using the description of the circle action on the replete bar construction $B^{\rm rep}(\langle z \rangle)$ provided in \cite[Proposition 3.21]{Rog09}. This is spelled out below.

Recall (from e.g.\ \cite[Proposition 3.21]{Rog09}) that the cyclic bar construction $B^{\rm cyc}(\langle z \rangle)$ decomposes as \[B^{\rm cyc}(\langle z \rangle) \simeq * \sqcup (\bigsqcup_{j \ge 1} S^1(j));\] i.e.\ as the subobject of $S^1 \times {\Bbb Z}$ given by the union of $S^1 \times {\Bbb N}_{> 0}$ and the zero element $(0, 1)$. For the replete bar construction, there is another circle in ``degree zero'', that is, \[B^{\rm rep}(\langle  z \rangle) \simeq \bigsqcup_{j \ge 0} S^1(j) = S^1 \times {\Bbb N}.\] The circle action extends in the natural manner: $t.(s, n) = (t^ns, n)$. In particular, the action on the degree zero circle $S^1(0) = S^1 \times \{0\}$ is trivial.   

By Construction \ref{constr:cycl_structure}, the replete Frobenius $\varphi_p \colon {\Bbb S}[B^{\rm rep}(\langle z \rangle)] \to {\Bbb S}[B^{\rm rep}(\langle z \rangle)]^{tC_p}$ factors through ${\Bbb S}[\psi_p] \colon {\Bbb S}[B^{\rm rep}(\langle z \rangle)] \to {\Bbb S}[B^{\rm rep}(\langle z \rangle)^{hC_p}]$, where $\psi_p$ is given by $(s, n) \mapsto (s^p, pn)$. In particular, the resulting diagram \[\begin{tikzcd}{\Bbb S}[B^{\rm rep}(\langle z \rangle)] \ar{d} \ar{r}{\varphi_p} & {
\Bbb S}[B^{\rm rep}(\langle z \rangle)]^{tC_p} \ar{d} \\ {\Bbb S}[z] \ar{r}{z \mapsto z^p} & {\Bbb S}[z]^{tC_p}\end{tikzcd}\]is a commutative diagram of $S^1$-equivariant ${\Bbb E}_{\infty}$-rings. The only part of this statement which is not covered by \cite[Proposition 11.3]{BMS19} is that its ``degree zero'' part \begin{equation}\label{logpointcyclotomicbase}\begin{tikzcd}{\Bbb S}[S^1(0)] \ar{r}{\varphi_p} \ar{d} &  {\Bbb S}[S^1(0)]^{tC_p} \ar{d} \\ {\Bbb S} \ar{r} & {\Bbb S}^{tC_p}\end{tikzcd}\end{equation} commutes, which is formal.
\end{proof}

\begin{remark} Commutativity of the diagram \eqref{logpointcyclotomicbase} is closely related to the ``spherical log point'' $({\Bbb S}, \langle z \rangle)$ (with structure map ${\Bbb S}[z] \to {\Bbb S}, z \mapsto 0$) being a log cyclotomic base. This is important if one wants to pursue the log crystalline comparison with this setup. 
\end{remark}

\subsection{Reducing to the perfectoid case} The following are analogs of \cite[Proposition 11.7 and Corollary 11.8]{BMS19}. Let us write $\langle z^{1/p^{\infty}} \rangle$ for the monoid $\frac{1}{p}{\Bbb N}$, so that ${\Bbb Z}[\langle z^{1 / p^{\infty}} \rangle] = {\Bbb Z}[z^{1/p^{\infty}}]$.  \begin{proposition}\label{prop:augmeq} The augmentation \begin{equation}\label{eq:augmeq1}{\rm THH}({\Bbb S}[z^{1/p^{\infty}}], \langle z^{1/p^{\infty}} \rangle) \to {\Bbb S}[z^{1/p^{\infty}}]\end{equation} is an equivalence after $p$-completion. 
\end{proposition}

\begin{proof} As in \cite[Proof of Proposition 11.7]{BMS19}, we can apply ${\rm THH}({\Bbb Z}) \otimes_{{\Bbb S}} - \simeq {\rm THH}({\Bbb Z}) \otimes_{{\rm THH}({\Bbb S})} -$ to \eqref{eq:augmeq1} to reduce to checking the statement for the augmentation \[{\rm HH}({\Bbb Z}[z^{1/p^{\infty}}], \langle z^{1/p^{\infty}} \rangle) \to {\Bbb Z}[z^{1/p^{\infty}}].\] This, in turn, follows from the log {\rm HKR}-filtration (Theorem \ref{thm:loghkrfilt}) and the vanishing of the log cotangent complex ${\Bbb L}_{({\Bbb Z}[z^{1/p^{\infty}}], \langle z^{1/p^{\infty}} \rangle)/{\Bbb Z}}$ after $p$-completion (Lemma \ref{lem:perfectvanishing}). 
\end{proof}

We now aim to describe an analogue of \cite[Corollary 11.8]{BMS19}, which states that, after adjoining all $p$-power roots of a uniformizer to an ${\cal O}_K$-algebra $A$, the absolute topological Hochschild homology ${\rm THH}(A[\bar{\omega}^{1/p^{\infty}}]; {\Bbb Z}_p)$ coincides with the $p$-completion of the relative term   ${\rm THH}(A[\bar{\omega}^{1/p^{\infty}}] / {\Bbb S}[z^{1/p^{\infty}}])$.  

Let us fix some notation to state the analogous result in the logarithmic context. If $(A, M)$ is an $({\cal O}_K, \langle \bar{\omega} \rangle)$-algebra, we can also consider $(A, M)$ an algebra over $({\Bbb S}[z], \langle z \rangle)$ by sending $z$ to $\bar{\omega}$. We denote by $M[z^{1/p^{\infty}}]$ the monoid obtained as the pushout of the diagram $M \xleftarrow{} \langle z \rangle \xrightarrow{} \langle z^{1/p^{\infty}} \rangle$. 

\begin{corollary} There are equivalences \begin{align*}&{\rm THH}((A, M)/({\Bbb S}[z], \langle z \rangle)) \otimes_{{\Bbb S}[z]} {\Bbb S}[z^{1/p^{\infty}}] \\ \simeq \quad & {\rm THH}((A[\bar{\omega}^{1/p^{\infty}}], M[z^{1/p^{\infty}}]) / ({\Bbb S}[z^{1/p^{\infty}}], \langle z^{1/p^{\infty}} \rangle))\end{align*} and \[{\rm THH}((A[\bar{\omega}^{1/p^{\infty}}], M[z^{1/p^{\infty}}]) ; {\Bbb Z}_p) \simeq {\rm THH}((A \otimes_{{\cal O}_K} {\cal O}_{K_{\infty}}, M[z^{1/p^{\infty}}]); {\Bbb Z}_p).\] After $p$-completion, all terms agree compatibly with the circle action and Frobenius. 
\end{corollary}

\begin{proof} The equivalence \[{\rm THH}((A, M)/({\Bbb S}[z], \langle z \rangle)) \otimes_{{\Bbb S}[z]} {\Bbb S}[z^{1/p^{\infty}}] \simeq {\rm THH}(A, M) \otimes_{{\rm THH}({\Bbb S}[z], \langle z \rangle)} {\Bbb S}[z^{1/p^{\infty}}]\] obtained by applying Lemma \ref{lem:transitivitylogthh} and the fact that the map ${\rm THH}({\Bbb S}[z], \langle z \rangle) \to {\Bbb S}[z^{1/p^{\infty}}]$ factors through ${\rm THH}({\Bbb S}[z^{1/p^{\infty}}], \langle z^{1/p^{\infty}} \rangle)$ readily implies the first equivalence by another application of Lemma \ref{lem:transitivitylogthh}. Now Proposition \ref{prop:augmeq} implies that this coincides with ${\rm THH}((A[\bar{\omega}^{1/p^{\infty}}], M[z^{1/p^{\infty}}]) ; {\Bbb Z}_p)$ after $p$-completion, and the last equivalence is clear. 
\end{proof}

\subsection{The complex $\widehat{\Prism}_{(A, M) / (W[k][[z]], \langle z \rangle)}$} We now define the complex  $\widehat{\Prism}_{(A, M) / (W[k][[z]], \langle z \rangle)}$ appearing in the statement of Theorem \ref{thm:derhamcomparison}. Since ${\rm THH}(({\cal O}_K, \langle \pi \rangle) / ({\Bbb S}[z], \langle z \rangle) ; {\Bbb Z}_p) \simeq {\rm THH}({\cal O}_K / {\Bbb S}[z] ; {\Bbb Z}_p)$, we obtain \begin{equation}\label{reltpeven}\pi_*{\rm TP}(({\cal O}_K, \langle \pi \rangle) / ({\Bbb S}[z], \langle z \rangle) ; {\Bbb Z}_p) \cong (W(k)[[z]])[\sigma^{\pm 1}]\end{equation} for a degree two class $\sigma$ by \cite[Proposition 11.10]{BMS19}. From this, we obtain the following analog of \cite[Proposition 11.11]{BMS19} (with analogous proof):

\begin{proposition} Let $(S, Q) \in {\rm lQRSPerfd}_{({\cal O}_K, \langle \pi \rangle)}$ be log quasiregular semiperfectoid. Then the presheaf $(S, Q) \mapsto \pi_0 {\rm TP}((S, Q) / ({\Bbb S}[z], \langle z \rangle) ; {\Bbb Z}_p)$ is a sheaf. 
\end{proposition}

\begin{proof} This follows from the discussion in Section \ref{subsec:nygaardcomplete}, using \eqref{reltpeven} and that $(S, Q)$ being log quasiregular semiperfectoid implies that $(S \widehat{\otimes}_{{\cal O}_K} {\cal O}_{K_{\infty}}, M[z^{1/p^{\infty}}])$ is also log quasiregular semiperfectoid.  
\end{proof}

We write ${\rm gr}^0 {\rm TP}((-, -) / ({\Bbb S}[z], \langle z \rangle) ; {\Bbb Z}_p)$ for the unfolding of the sheaf $\pi_0{\rm TP}((-, -) / ({\Bbb S}[z], \langle z \rangle) ; {\Bbb Z}_p)$ on ${\rm lQRSPerfd}_{({\cal O}_K, \langle \pi \rangle)}$ to ${\rm lQSyn}_{({\cal O}_K, \langle \pi \rangle)}$.

\begin{definition} For $(A, M) \in {\rm lQSyn}_{({\cal O}_K, \langle \pi \rangle)}$, we write \[\widehat{\Prism}_{(A, M) / (W(k)[[z]], \langle z \rangle))} := {\rm gr}^0 {\rm TP}((A, M) / ({\Bbb S}[z], \langle z \rangle) ; {\Bbb Z}_p)\] for the value of the sheaf ${\rm gr}^0 {\rm TP}((-, -) / ({\Bbb S}[z], \langle z \rangle) ; {\Bbb Z}_p)$ at $(A, M)$.
\end{definition}

\begin{proof}[Proof of Theorem \ref{thm:derhamcomparison}] This is now analogous to \cite[Proof of Corollary 11.12(ii)]{BMS19}. There is a canonical equivalence \[{\rm gr}^0 {\rm TP}((A, M) / ({\Bbb S}[z], \langle z \rangle) ; {\Bbb Z}_p)/E(z) \simeq {\rm gr}^0 {\rm HP}((A, M) / ({\cal O}_K, \langle \pi \rangle) ; {\Bbb Z}_p)\] by the analogous statement for $(A, M) \in {\rm lQRSPerfd}_{({\cal O}_K, \langle \pi \rangle)}$. The result now follows from  Theorem \ref{thm:logantieau} and \cite[Corollary 7.6]{Bha12}.
\end{proof}

\bibliographystyle{amsalpha}

\bibliography{LogTHHPrismatic_final}

\end{document}